\documentclass{amsart}

\usepackage{enumerate,amssymb,amsmath,mathrsfs}
\usepackage{amssymb,amsfonts,amsthm,amscd}
\usepackage[mathscr]{eucal}     
\usepackage{graphicx}
\usepackage[arrow, matrix, curve]{xy}
\usepackage{color}
\definecolor{gray}{gray}{.75}
\definecolor{gray2}{gray}{.50}

\newcommand{\one}{\mathbf{1}}
\newcommand{\dom}{\mathcal{D}}

\newcommand{\domrj}{\mathcal{D}_{j,\min}}

\newcommand{\domaj}{\mathcal{D}_{j,\max}}
\newcommand{\domm}{\widetilde{\mathcal{D}}}

\newcommand{\B}{\mathcal{B}}

\newcommand{\Z}{\mathbb{Z}}
\newcommand{\N}{\mathbb{N}}
\newcommand{\R}{\mathbb{R}}
\newcommand{\C}{\mathbb{C}}

\newcommand{\D}{\nabla}

\newcommand{\DD}{\widetilde{\nabla}}
\newcommand{\G}{\Gamma}
\newcommand{\GG}{\widetilde{\Gamma}}

\newcommand{\A}{\mathrm{\alpha}}
\newcommand{\w}{\mathrm{\omega}}

\newtheorem{thm}{Theorem}[section]
\newtheorem{cor}[thm]{Corollary}
\newtheorem{remark}[thm]{Remark}
\newtheorem{prop}[thm]{Proposition}
\newtheorem{lemma}[thm]{Lemma}

\numberwithin{equation}{section}

\begin{document}

\title[Gluing Formula for Refined Analytic Torsion]
{Gluing Formula for Refined Analytic Torsion}
\author{Boris Vertman}
\address{University of Bonn \\ Department of Mathematics \\
Beringstr. 6\\ 53115 Bonn\\ Germany}
\email{vertman@math.uni-bonn.de}


\begin{abstract}
{In a previous article we have presented a construction of refined analytic torsion in the spirit of Braverman and Kappeler, which does apply to compact manifolds with and without boundary. We now derive a gluing formula for our construction, which can be viewed as a gluing law for the original definition of refined analytic torsion by Braverman and Kappeler.}
\end{abstract}

\maketitle

\pagestyle{myheadings}
\markboth{\textsc{Gluing Formula}}{\textsc{Boris Vertman}} 

\section{Introduction}\
\\[-3mm] In this publication we turn to the main motivation for the proposed construction of refined analytic torsion $-$ a gluing formula. A gluing formula allows to compute the torsion invariant by cutting the manifold into elementary pieces and performing computations on each component. Certainly, the general fact of existence of such gluing formulas is remarkable from the analytic point of view, since the secondary spectral invariants are uppermost non-local.
\\[3mm] We establish a gluing formula for the refined analytic torsion in three steps. First we establish a splitting formula for the eta-invariant of the even part of the odd-signature operator. This is essentially an application of the results in [KL].
\\[3mm] Secondly we establish a splitting formula for the refined torsion $\rho_{[0,\lambda]}$ in the special case $\lambda =0$. This is the most intricate part and is done by a careful analysis of long exact sequences in cohomology und the Poincare duality on manifolds with boundary. The discussion is subdivided into several sections.
\\[3mm] Finally we are in the position to establish the desired gluing formula for the refined analytic torsion, as a consequence of the Cheeger-M\"{u}ller Theorem and a gluing formula for the combinatorial torsion by M. Lesch [L2]. As a byproduct we also obtain a splitting formula for the scalar analytic torsion in terms of combinatorial torsion of a long exact sequence on cohomology, refining the result of Y. Lee in [Lee, Theorem 1.7 (2)].
\\[3mm] In our discussion we do not rely on the gluing formula of S. Vishik in [V], where only the case of trivial representations is treated. In particular we use a different isomorphism between the determinant lines, which is more convenient in the present setup.
\\[3mm] We perform the proof under the assumption of a flat Hermitian metric, in other words in case of unitary representations. This is done partly because the Cheeger-M\"{u}ller Theorem for manifolds with boundary and unimodular representations is not explicitly established for the time being. It seems, however, that the appropriate result can be established by an adaptation of arguments in [L\"{u}] and [Mu]. 
\\[3mm] Finally it should be emphasized that the presented result can also be viewed as a gluing formula for the refined analytic torsion in the sense of Braverman and Kappeler.
\\[3mm] {\bf Acknowledgements.} The results of this article were obtained during the author's Ph.D. studies at Bonn University, Germany. The author would like to thank his thesis advisor Prof. Matthias Lesch for his support and useful suggestions. The author is also grateful to Prof. Werner M\"{u}ller for helpful discussions. The author was supported by the German Research Foundation as a scholar of the Graduiertenkolleg 1269 "Global Structures in Geometry and Analysis".

\section{Setup for the Gluing Formula}\label{gluing-statement} \
\\[-3mm] Let $M=M_1\cup_{N}M_2$ be an odd-dimensional oriented closed Riemannian manifold where $N$ is an embedded closed hypersurface of codimension one which separates $M$ into two pieces $M_1$ and $M_2$ such that $M_j, j=1,2$ are compact bounded Riemannian manifolds with $\partial M_j=N$ and orientations induced from $M$. The setup is visualized in the Figure \ref{split-manifold-picture} below:
\begin{figure}[h]
\begin{center}
\includegraphics[width=0.5\textwidth]{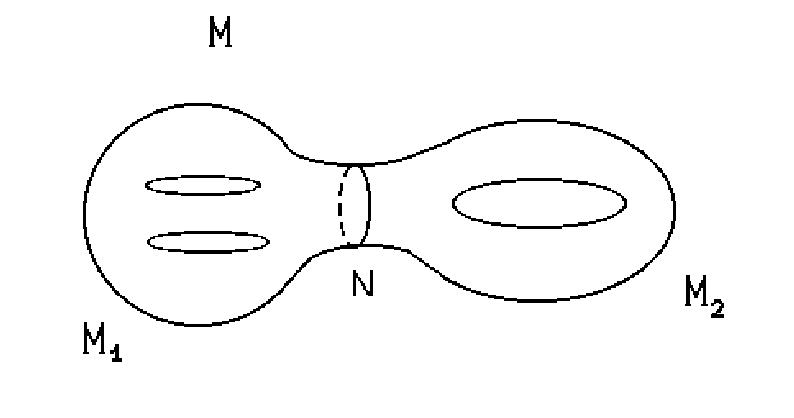}
\caption{A compact closed split-manifold $M=M_1\cup_{N}M_2$.}
\label{split-manifold-picture}
\end{center}
\end{figure}
\\[-2mm] Let $\rho:\pi_1(M)\to U(n,\C)$ be a unitary representation of the fundamental group of $M$. Denote by $\widetilde{M}$ the universal cover of $M$. It is a principal bundle over $M$ with the structure group $\pi_1(M)$, cf. [KN, Proposition 5.9 (2)]. Consider the complex vector bundle $E$ associated to the principal bundle $\widetilde{M}$ via the representation $\rho$. $$E=\widetilde{M}\times_{\rho}\C^n.$$ 
The vector bundle is naturally endowed with a canonical flat connection $\D$, induced by the exterior derivative on $\widetilde{M}$. The holonomy representation of $\D$ is given by the representation $\rho$. 
\\[3mm] Note that all flat vector bundles arise this way. In fact flatness of a given connection on a vector bundle implies that the associated holonomy map gives rise to a well-defined representation of the fundamental group of the base manifold, and the representation is related to the vector bundle as above.
\\[3mm] By unitariness of the representation $\rho$ the standard Hermitian inner product on $\C^n$ gives rise to a smooth Hermitian metric $h^E$ on $E$, compatible with the flat connection $\D$. In other words the canonically induced Hermitian metric $h^E$ is flat.
\\[3mm] Assume the metric structures $(g^M, h^E)$ to be product near the hypersurface $N$. The issues related to the product structures are discussed in detail in [BLZ, Section 2]. More precisely, we identify using the inward geodesic flow an open collar neighborhood $U\subset M$ of the hypersurface $N$ diffeomorphically with $(-\epsilon,\epsilon)\times N, \epsilon > 0$, where the hypersurface $N$ is identified with $\{0\}\times N$. The metric $g^M$ is product over the collar neighborhood of $N$, if over $U$ it is given under the diffeomorphism $\phi: U \to (-\epsilon,\epsilon)\times \partial M$ by 
\begin{align}\label{Riemann-product}
\phi_*g^M|_U=dx^2\oplus g^M|_{N}.
\end{align}
The diffeomorphism $U \cong (-\epsilon,\epsilon)\times N$ shall be covered by a bundle isomorphism $\widetilde{\phi}: E|_U \to (-\epsilon,\epsilon)\times E|_N$. The fiber metric $h^E$ is product near the boundary, if it is preserved by the bundle isomorphism, i.e. if for all $x\in (-\epsilon, \epsilon)$
\begin{align}\label{Hermitian-product}
\widetilde{\phi}_*h^E|_{\{x\}\times N}=h^E|_{N}.
\end{align}
The restrictive assumption of product metric structures is necessary to apply the splitting formula of [KL] to our setup, which works only on Dirac type operators in product form over the collar with constant tangential part.
\\[3mm] Furthermore we use the product metric structures in order to apply the Cheeger-M\"{u}ller Theorem for manifolds with boundary, ( cf. [L\"{u}], [V]). However with the anomaly formulas in [BZ1] and [BM] the product structures are not essential here. 
\\[3mm] By Leibniz rule the connection $\D$ gives rise to flat twisted exterior differential on smooth $E$-valued differential forms. The restrictions of $(E,\D)$ to $M_j,j=1,2$ give rise to twisted de Rham complexes $(\Omega_0^*(M_j,E),\D_j)$. We denote their minimal and maximal extensions by $$(\mathcal{D}_{j, \min /\max}, \D_{j, \min /\max}), $$ respectively. By [BV3, Theorem 3.2], which is in the untwisted setup essentially the statement of [BL1, Theorem 4.1], these complexes are Fredholm and their cohomologies can be computed from smooth subcomplexes as follows. Consider for $j=1,2$ the natural inclusions $\iota_j:N\hookrightarrow M_j$ and put
\begin{align*}
&\Omega^*_{\textup{min}}(M_j,E):=\{\w \in \Omega^*(M,E)| \iota^*_j\w=0\},\\
&\Omega^*_{\textup{max}}(M_j,E):= \Omega^*(M,E).
\end{align*}
The operators $\D_j$ yield exterior derivatives on $\Omega^*_{\textup{min}}(M_j,E)$ and $\Omega^*_{\textup{max}}(M_j,E)$. The complexes $(\Omega^*_{\textup{min / max}}(M_j,E),\D_j)$ are by [BV3, Theorem 3.2] smooth subcomplexes of the Fredholm complexes $(\mathcal{D}_{j, \min /\max}, \D_{j, \min /\max})$ with
\begin{align*}
H^*_{\textup{rel / abs}}(M_j,E)&:=H^*(\Omega^*_{\min / \max }(M_j,E),\D_j)\\ &\cong H^*(\mathcal{D}_{j, \min /\max}, \D_{j, \min /\max}), \quad j=1,2.
\end{align*}
Finally we define for $j=1,2$
\begin{align*}
(\domm_j,\DD_j)&:=(\domrj, \D_{j, \min}) \oplus (\domaj, \D_{j, \max}), \\
(\domm,\DD)&:=(\dom, \D) \oplus (\dom, \D),
\end{align*}
where $(\dom, \D)$ denotes the unique ideal boundary conditions of $(\Omega^*(M, E),\D)$. 
\\[3mm] Furthermore, the Riemannian metric $g^M$ and the fixed orientation on $M$ give rise to the Hodge-star operator for any $k=0,..,m=\dim M$: $$*:\Omega^k(M,E)\to \Omega^{m-k}(M,E).$$ Define $$\G :=i^r(-1)^{\frac{k(k+1)}{2}}*:\Omega^k(M,E)\to \Omega^{m-k}(M,E), \quad r:= (\dim M+1)/2.$$ This operator extends to a well-defined self-adjoint involution on the $L^2-$completion of $\Omega^*(M,E)$, which we also denote by $\G$. We indicate the restriction of $\G$ to $M_j,j=1,2$ by the additional subscript $j$ and define, corresponding to [BV3, Definition 3.5], the odd-signature operators of the complexes $(\domm, \DD)$ and $(\domm_j,\DD_j), j=1,2$: 
\begin{align*}
\B:&=\GG\DD+\DD\GG, \\
\B^j&:=\GG_j\DD_j+\DD_j\GG_j, j=1,2.
\end{align*}
The upper index $j$ will not pose any confusion with the square of the odd signature operator, since the squared odd-signature operator does not appear in the arguments below. The presented notation remains fixed throughout the discussion below, unless stated otherwise. However for convenience, some setup and notation will be repeated for clarification.

\section{Temporal Gauge Transformation}\label{appendix-gauge}
Consider the closed oriented Riemannian split-manifold $(M,g^M)$ and the flat Hermitian vector bundle $(E,\D,h^E)$ with the structure group $G:=U(n,\C)$ as introduced in Section \ref{gluing-statement}. Denote the principal $G$-bundle associated to $E$ by $P$. $G$ acts on $P$ from the right.
\\[3mm] Consider $U\cong (-\epsilon,\epsilon)\times N$ the collar neighborhood of the splitting hypersurface $N$. We view the restrictions $P|_U,P|_{N}$ as $G$-bundles, where the structure group can possibly be reduced to a subgroup of $G$. 
\\[3mm] By the setup of Section \ref{gluing-statement} the bundle structures are product over $U$. More precisely let $\pi: (-\epsilon, \epsilon)\times \partial X \to \partial X$ be the natural projection onto the second component. We have a bundle isomorphisms $E|_U\cong \pi^*E|_{N}$ and for the associated principal bundles 
\begin{align*}
P|_U\cong \pi^*P|_{N}\xrightarrow{f} P|_{N},
\end{align*}
where $f$ is the principal bundle homomorphism, covering $\pi$, with the associated homomorphism of the structure groups $G \to G$ being the identity automorphism. 
\\[3mm] Now let $\w_{N}$ denote a flat connection one-form on $P|_{N}$. Then $$\w_U:=f^*\w_{N}$$ gives a connection one-form on $P|_U$ which is flat again. 
\\[3mm] In order to understand the structure of $\w_U=f^*\w_{N}$, let $\{\widetilde{U}_{\A},\widetilde{\Phi}_{\A}\}$ be a set of local trivializations for $P|_{N}$. Then $P|_U\cong \pi^*P|_{N}$ is equipped with a set of naturally induced local trivializations $\{U_{\A}:=(-\epsilon, \epsilon)\times \widetilde{U}_{\A}, \Phi_{\A}\}$. The local trivializations define local sections $\widetilde{s}_{\A}$ and $s_{\A}$ as follows. For any $p \in \widetilde{U}_{\A}$, normal variable $x \in (-\epsilon, \epsilon)$ and for $e\in G$ being the identity matrix we put
\begin{align*}
\widetilde{s}_{\A}(p):=\widetilde{\Phi}_{\A}^{-1}(p,e), \\
s_{\A}(x,p):=\Phi_{\A}^{-1}((x,p),e).
\end{align*}
We use the local sections to obtain local representations for the connection one-forms $\w_U$ and $\w_{N}$:
\begin{align*}
\widetilde{\w}^{\A}:=\widetilde{s}^*_{\A}\w_{N}\in \Omega^1(\widetilde{U}_{\A},\mathcal{G}), \\
\w^{\A}:=s^*_{\A}\w_U\in \Omega^1(U_{\A},\mathcal{G}),
\end{align*}
where $\mathcal{G}$ denotes the Lie Algebra of $G$. By construction both local representations are related as follows. Let $(x,\underline{y})=(x,y_1,..,y_n)$ denote the local coordinates on $U_{\A}=(-\epsilon, \epsilon)\times \widetilde{U}_{\A}$ with $x \in (-\epsilon, \epsilon)$ being the normal coordinate and $\underline{y}$ denoting the local coordinates on $\widetilde{U}_{\A}$. Then 
\begin{align}
\widetilde{\w}^{\A}&=\sum\limits_{i=1}^n \w^{\A}_i(\underline{y})dy_i, \nonumber \\
\w^{\A}&=\w^{\A}_0(x,\underline{y})dx+\sum\limits_{i=1}^n \w^{\A}_i(x,\underline{y})dy_i, \nonumber \\
\textup{with} \quad \w^{\A}_0&\equiv 0, \ \textup{and} \ \w^{\A}_i(x,\underline{y})\equiv \w^{\A}_i(\underline{y}). \label{temp-stern}
\end{align}
We call a flat connection $\w$ on $P$ a connection in temporal gauge, if there exists a flat connection $\w_N$ on $P_N$ such that over the collar neighborhood $U$ $$\w|_U=f^*\w_{N}.$$ The local properties of a connection in temporal gauge and in particular its independence of the normal variable $x\in (-\epsilon, \epsilon)$ are discussed in \eqref{temp-stern}. Our aim is to show that any flat connection one-form on a principal bundle $P$ can be gauge transformed to a flat connection in temporal gauge. Recall that a gauge-transformation of $P$ is a principal bundle automorphism $g \in \textup{Aut}(P)$ covering identity on $M$ with $g(p\cdot u)=g(p)\cdot u$ for any $u\in G$ and $p\in P$.
\\[3mm] A gauge transformation can be viewed interchangedly as a transformation from one system of local trivializations into another. Hence the action of a gauge transformation on a connection one-form is determined by the transformation law for connections under change of coordinates.
\\[3mm] We have the following result.
\begin{prop}\label{t-g}
Any flat connection on the principal bundle $P$ is gauge equivalent to a flat connection in temporal gauge.
\end{prop} 
\begin{proof}
By a partition of unity argument it suffices to discuss the problem locally over a trivializing neighborhood $(U_{\A}:=(-\epsilon, \epsilon)\times \widetilde{U}_{\A}, \Phi_{\A})$.
\\[3mm] Let $\w$ be a flat connection on $P|_U$. Let $g$ be any gauge transformation on $P|_U$. Denote the gauge transform of $\w$ under $g$ by $\w_g$. 
\\[3mm] Over the trivializing neighborhood $U_{\A}$ the connections $\w,\w_g$ and the gauge tranformation $g$ are given by local $\mathcal{G}$-valued one-forms $\w^{\A}, \w_g^{\A}$ and a $G$-automorphism $g^{\A}$ respectively. They are related in correspondence to the transformation law of connections as follows
$$\w^{\A}_g=(g^{\A})^{-1}\circ \w^{\A}\circ g^{\A}+(g^{\A})^{-1}dg^{\A},$$
where the action $\circ$ is the concatenation of matrices ($G\subset GL(n,\C)$), after evaluation at a local vector field and a base point in $U_{\A}$. The local one form $\w^{\A}$ writes as
$$\w^{\A}=\w^{\A}_0(x,\underline{y})dx+\sum\limits_{i=1}^n w^{\A}_i(x,\underline{y})dy_i.$$
In order to gauge-transform $\w$ into temporal gauge, we need to annihilate $\w^{\A}_0$ and the $x$-dependence in $\w^{\A}_i$. For this reason we consider the following initial value problem with parameter $\underline{y}\in \widetilde{U}_{\A}$
\begin{align}\nonumber \partial_xg^{\A}(x,\underline{y})=-\w^{\A}_0(x,\underline{y})g^{\A}(x,\underline{y}), \\
g^{\A}(0,\underline{y})=\one \in GL(n,\C). \label{ODE}
\end{align}
In order to identify the solution to \eqref{ODE} consider for any fixed $\underline{y}\in \widetilde{U}_{\A}$ the following $x-$time dependent vector field $V^{\A}_{x,\underline{y}}, x\in (-\epsilon, \epsilon)$ on $G$: $$\forall u\in G \quad V^{\A}_{x,\underline{y}}u:=-(R_u)_* \w^{\A}_0(x,\underline{y})= -\w^{\A}_0(x,\underline{y})\cdot u,$$
where $R_u$ is the right multiplication on $G$ and the second equality follows from the fact that $G\in GL(n,\C)$ is a matrix Lie group.
\\[3mm] Let $\widetilde{g}^{\A}(x, \underline{y})$ be the unique integral curve of the time-dependent vector field $V^{\A}_{x,\underline{y}}$ with $\widetilde{g}^{\A}(0, \underline{y})=\one\in G$. It satisfies $$\partial_x \widetilde{g}^{\A}(x, \underline{y})=V^{\A}_{x,\underline{y}}\widetilde{g}^{\A}(x, \underline{y})=-\w^{\A}_0(x,\underline{y})\widetilde{g}^{\A}(x, \underline{y}).$$
Hence the integral curve $\widetilde{g}^{\A}(x, \underline{y})$ solves \eqref{ODE}. By the fundamental theorem for ordinary linear differential equations (cf. [KN, Appendix 1]) we know that the initial value problem \eqref{ODE} has a unique solution, smooth in $x\in (-\epsilon, \epsilon)$ and $\underline{y}\in \widetilde{U}_{\A}$. Since $\widetilde{g}^{\A}(x,\underline{y})$ solves \eqref{ODE} we find that the solution is moreover $G-$valued.
\\[3mm] With gauge transformation $g$ being locally the solution to \eqref{ODE} we find for the gauge transformed connection $\w_g$
\begin{align*}
\w^{\A}_g=(g^{\A})^{-1}\circ \w^{\A}\circ g^{\A}+(g^{\A})^{-1}dg^{\A}=\\
=(g^{\A})^{-1}\circ \w^{\A}_0\circ g^{\A}dx+\sum\limits_{i=1}^n (g^{\A})^{-1}\circ \w^{\A}_i\circ g^{\A}dy_i + \\ + (g^{\A})^{-1}\partial_x g^{\A}dx +\sum\limits_{i=1}^n (g^{\A})^{-1}\partial_{y_i} g^{\A}dy_i= \\
=\sum\limits_{i=1}^n (g^{\A})^{-1}\circ \w^{\A}_i\circ g^{\A}dy_i + \sum\limits_{i=1}^n (g^{\A})^{-1}\partial_{y_i} g^{\A}dy_i.
\end{align*} 
where in the last equality we cancelled two summands due to $g^{\A}$ being the solution to \eqref{ODE}. So far we didn't use the fact that $\w$ is a flat connection. A gauge transformation preserves flatness, so $\w_g$ is flat again. Put
$$\w^{\A}_g=\w^{\A}_{g,0}(x,\underline{y})dx+\sum\limits_{i=1}^n \w^{\A}_{g,i}(x,\underline{y})dy_i,
$$ where by the previous calculation 
$$\w^{\A}_{g,0}\equiv 0, \quad \w^{\A}_{g,i}\equiv (g^{\A})^{-1}\circ \w^{\A}_i\circ g^{\A}+ (g^{\A})^{-1}\partial_{y_i} g^{\A}.$$
Flatness of $\w_g$ implies $$\partial_x\w^{\A}_{g,i}(x,\underline{y})=\partial_{y_i}\w^{\A}_{g,0}(x,\underline{y})=0.$$
Hence the gauge transformed connection is indeed in temporal gauge. This completes the proof.
\end{proof}\ \\
\\[-7mm] A gauge transformation, viewed so far as a principal bundle automorphism on the $G-$principal bundle $P$, can equivalently be viewed as a $G-$valued bundle automorphism on the vector bundle $E$ associated to $P$. We adopt this point of view for the forthcoming discussion.
\\[3mm] Take the given flat connection $\D$ on the Hermitian vector bundle $(E,h^E)$ with the structure group $G=U(n,\C)$ and the canonical metric $h^E$ induced by the standard inner product on $\C^n$. Proposition \ref{t-g} asserts existence of a temporal gauge transformation $g \in Aut_G(E)$ such that the gauge transformed covariant derivative $g\D g^{-1}$ is in temporal gauge (a covariant derivative is said to be in temporal gauge if the associated connection one-form is in temporal gauge). 
\\[3mm] The temporal gauge transformation $g$ gives rise to a map on sections in a natural way $$\mathfrak{G}: \Omega^*(M, E\otimes E)\rightarrow \Omega^*(M, E \otimes E).$$ Due to the fact that $g$ takes locally values in $U(n,\C)$ and the Hermitian metric $h^E$ is canonically induced by the standard inner product on $\C^n$, we obtain the following result:
\begin{prop}
$\mathfrak{G}$ extends to a unitary transformation  $$\mathfrak{G}: L^2_*(M, E\otimes E, g^M, h^E)\rightarrow L^2_*(M, E\otimes E, g^M, h^E).$$
\end{prop}
\begin{cor}\label{frak-G}
The odd-signature operators $\B=\B(\D)$ and $\B^j=\B^j(\D_j), j=1,2$ are spectrally equivalent to $\B(g\D g^{-1})$ and $\B^j(g\D g^{-1}|_{M_j}), j=1,2$ respectively.
\end{cor}\ \\
\\[-7mm] The statement of the corollary above follows from invariance of minimal and maximal extensions under unitary transformations and from the fact that unitary transformations preserve spectral properties of operators, compare also [BV1, Proposition 4.2]
\\[3mm] The statement of the corollary implies that in the setup of this section (for unitary vector bundles) the assumption of temporal gauge is done without loss of generality, which we do henceforth. In this particular geometric setup we obtain the following specific result for refined analytic torsion.
\begin{prop}\label{rho-formula} Let $T^{RS}(\DD)$ and $T^{RS}(\DD_j), j=1,2$ denote the scalar analytic torsions associated to the complexes $(\domm, \DD), (\domm_j, \DD_j)$, respectively. Furthermore let $\rho_{\G}(M,E)$ and $\rho_{\G}(M_j,E)$ denote the associated refined torsion elements in the sense of [BV3, (3.7)] for $\lambda=0$. Then we have
\begin{align*}
\rho_{\textup{an}}(\D)=\frac{1}{T^{RS}(\DD)}\cdot \exp \left[-i\pi \eta(\B_{\textup{even}})+i\pi \textup{rk}(E) \eta (\B_{\textup{trivial}})\right] \times  \hspace{30mm}\\
\times \exp \left[-i\pi \frac{m-1}{2}\dim \ker \B_{\textup{even}}+i\pi \textup{rk}(E) \frac{m}{2}\dim \ker \B_{\textup{trivial}}\right] \rho_{\G}(M,E),\\
\rho_{\textup{an}}(\D_j)=\frac{1}{T^{RS}(\DD_j)}\cdot \exp \left[-i\pi \eta(\B^j_{\textup{even}})+i\pi \textup{rk}(E) \eta (\B^j_{\textup{trivial}})\right] \times \hspace{30mm}\\
\times \exp \left[ -i\pi \frac{m-1}{2}\dim \ker \B^j_{\textup{even}}+i\pi \textup{rk}(E) \frac{m}{2}\dim \ker \B^j_{\textup{trivial}}\right] \rho_{\G}(M_j,E).
\end{align*} 
\end{prop}
\begin{proof}
Recall from the definition of refined analytic torsion in [BV3, Corollary 4.9]
\begin{align*}
\rho_{\textup{an}}(\D)&=e^{\xi_{\lambda}(\D)}\exp [-i\pi (\eta(\B_{\textup{even}}^{(\lambda, \infty)})+\xi'_{\lambda}(\D))] \times \\
&\exp [+i\pi \textup{rk}(E)(\eta(\B_{\textup{trivial}})+\xi'(\D_{\textup{trivial}}))]\cdot \rho_{[0,\lambda]}, \\
\rho_{\textup{an}}(\D_j)&=e^{\xi_{\lambda}(\D_j)}\exp [-i\pi (\eta(\B_{\textup{even}}^{j, (\lambda, \infty)})+\xi'_{\lambda}(\D_j))] \times \\
&\exp [+i\pi \textup{rk}(E)(\eta(\B^j_{\textup{trivial}})+\xi'(\D_{j, \textup{trivial}}))]\cdot \rho^j_{[0,\lambda]}, j=1,2.
\end{align*}
The assumption of product metric structures and the temporal gauge allow a reduction to closed double manifolds, as performed explicitly in [BV3, Theorem 5.3]. This yields by similar arguments, as in [BK2, Proposition 6.5]:
\begin{align*}
\xi'_{\lambda}(\D)&=\frac{1}{2}\sum\limits_{k=0}^m(-1)^{k+1}\cdot k \cdot \dim \domm^k_{[0,\lambda]}, \\
\xi'_{\lambda}(\D_j)&=\frac{1}{2}\sum\limits_{k=0}^m(-1)^{k+1}\cdot k \cdot \dim \domm^k_{j, [0,\lambda]}.
\end{align*}
Now, via [BV3, Lemma 4.8] we obtain
\begin{align*}
\xi'_{\lambda}(\D)&\equiv \frac{m}{2} \dim \domm^{\textup{even}}_{[0,\lambda]} \ \textup{mod} \ 2\Z, \\
\xi'_{\lambda}(\D_j)&\equiv \frac{m}{2} \dim \domm^{\textup{even}}_{j, [0,\lambda]} \ \textup{mod} \ 2\Z.
\end{align*}
Similar arguments show
\begin{align*}
\xi'(\D_{\textup{trivial}})\equiv \frac{m}{2}\dim \ker \B_{\textup{trivial}} \ \textup{mod} \ 2\Z, \\
\xi'(\D_{j, \textup{trivial}})\equiv \frac{m}{2}\dim \ker \B^j_{\textup{trivial}} \ \textup{mod} \ 2\Z.
\end{align*}
Fix $\lambda =0$ and observe for $j=1,2$ from [BV3, (4.9)]:
\begin{align*}
\eta (\B^{(0, \infty)}_{\textup{even}})+\xi'_0(\D)\equiv \eta (\B_{\textup{even}})+\frac{m-1}{2}\dim \ker \B_{\textup{even}} \ \textup{mod} \ 2\Z, \\
\eta (\B^{j, (0, \infty)}_{\textup{even}})+\xi'_0(\D_j)\equiv \eta (\B^j_{\textup{even}})+\frac{m-1}{2}\dim \ker \B^j_{\textup{even}} \ \textup{mod} \ 2\Z.
\end{align*}
Now the statement of the proposition follows from the fact that flatness of the Hermitian metric $h^E$ implies equality between the squared odd-signature operator and the Laplacians of the corresponding complexes, and hence 
\begin{align*}
e^{\xi_0(\D)}=\frac{1}{T^{RS}(\DD)}, \quad e^{\xi_0(\D_j)}=\frac{1}{T^{RS}(\DD_j)}.
\end{align*}
This proves the proposition. 
\end{proof}

\section{Splitting formula for the eta-invariant} \label{eta} \
\\[-4mm] This section is an application of [KL, Theorem 7.7]. For the setup of that result consider $\mathcal{U}\subset M$ the collar neighborhood of the hypersurface $N$ together with a mapping $$\Phi : C^{\infty}(\mathcal{U}, F|_{\mathcal{U}}) \to C^{\infty}((-\epsilon, \epsilon), C^{\infty}(N,F|_N)),$$
where $F$ is any Hermitian vector bundle over $M$ and $\Phi$ extends to an isometry on the $L^2-$completions of the spaces. Now let $D$ be a self-adjoint operator of Dirac-type with the following product form over the collar neighborhood: $$\Phi \circ D|_{\mathcal{U}}\circ \Phi^{-1}=\gamma \left[\frac{d}{dx}+A\right],$$
where $\gamma$ is a bundle homomorphism on $C^{\infty}(N,F|_N)$ and the tangential operator $A$ is a self-adjoint operator of Dirac-type over $C^{\infty}(N,F|_N)$. The essence of the product form lies in the $x$-independence of the coefficients $\gamma$ and $A$. Moreover we assume 
\begin{align}\label{comm-gamma-A}
\gamma^2=-I, \quad \gamma^*=-\gamma, \quad \gamma A=-A\gamma.
\end{align}
By restriction to $M_1,M_2$ we obtain Dirac operators $D^1,D^2$ with product structure (under the identification of $\Phi$) as above over the collar neighborhoods $\mathcal{U}\cap M_j$ of the boundaries $\partial M_j=N, j=1,2$.  
\\[3mm] Let $P:L^2(N,F|_N)\to L^2(N,F|_N)$ satisfy the following conditions:
\begin{align}
&\bullet \ \textup{P is pseudo-differential of order zero}, \label{boundary1} \\
&\bullet \ \textup{P is an orthogonal projection, i.e. $P^*=P, P^2=P$}, \label{boundary2}\\
&\bullet \ \textup{$\gamma P\gamma^*=I-P$}, \label{boundary3}\\
&\bullet \ \textup{$(P_{>0}, P)$ form a Fredholm pair, i.e. $P_{>0}|_{\textup{im}P}$ is Fredholm}. \label{boundary4}
\end{align}
Here $P_{>0}$ denotes the positive spectral projection associated to the self-adjoint tangential operator $A$. The boundary value problems $(D^j,P),j=1,2$ are well-posed in the sense of R. T. Seeley, by [BL3, Theorem 7.2]. We know, see [KL, Theorem 3.1] and the references therein, that the eta-functions $\eta(D^j_P,s)$ extend meromorphically to $\C$. Assume for simplicity that the eta-functions are regular at $s=0$ and set for $j=1,2$:
\begin{align*}
\eta(D^j_P)&:=\frac{1}{2}\left[\eta(D^j_P,s=0) + \dim \ker D^j_P\right], \\
\eta(D)&:=\frac{1}{2}\left[\eta(D,s=0) + \dim \ker D\right].
\end{align*}
This definition coincides with [BV3, (4.1)] for $D^j_P=\B^j_{\textup{even}}$ and $D=\B_{\textup{even}}$, since in the setup of the present section the odd-signature operators are self-adjoint, hence have real spectrum.
\\[3mm] The same holds for $D^j_{I-P}$ as well, and the splitting formula in the version of [KL, Theorem 7.7] is given as follows:
\begin{align}\label{splitting-eta-general}
\eta(D)=\eta(D^1_P)+\eta(D^2_{I-P})-\tau_{\mu}(I-P_1,P,P_1),
\end{align}
where $P_1$ denotes the Calderon projector for $D^1$, which is the orthogonal projection of sections in $F|_N$ onto the Cauchy-data space of $D^1$ consisting of the traces at $N$ of elements in the kernel of $D^1$. For further details see [BW].
\\[3mm] The third summand $\tau_{\mu}$ in \eqref{splitting-eta-general} refers to the Maslov triple index defined in [KL, Definition 6.8]. The Maslov triple index is integer-valued and thus the result above leads in particular to a mod $\Z$ splitting formula for eta-invariants.  
\\[3mm] Leaving for the moment these general constructions aside, we continue in the setup of the Section \ref{gluing-statement}. We adapt the constructions of [KL, Section 8.1] to the present situation. Let $\iota: N \hookrightarrow M$ denote the inclusion of the splitting hypersurface $N$ into the closed split-manifold $M$. Define the restriction map: 
\begin{align*}
R: \Omega^{\textup{even}}(M,E\oplus E)&\to \Omega^*(N, (E\oplus E)|_N), \\
\beta &\mapsto \iota^*(\beta) + \iota^*(\widetilde{*}\beta), 
\end{align*}
where $\widetilde{*}$ is the Hodge-star operator on the oriented Riemannian manifold $M$ acting antidiagonally on $E \oplus E$ with the following matrix form: 
\begin{align*}
\widetilde{*}=\left(\begin{array}{cc}0 & * \\ * & 0\end{array}\right).
\end{align*}
It is further related as follows to the chirality operator $\GG$ of the Hilbert complex $(\domm, \DD)$: 
\begin{align*}
\GG :=i^r(-1)^{\frac{k(k+1)}{2}}\widetilde{*}:\Omega^k(M,E\oplus E)\to \Omega^{m-k}(M,E\oplus E), \\ \textup{where} \ r:=(\dim M +1)/2.
\end{align*}
The restriction map $R$ induces with $\mathcal{U}\cong (-\epsilon, \epsilon)\times N$ the following identification: 
$$\Phi: \Omega^{\textup{even}}(\mathcal{U}, (E\oplus E)|_{\mathcal{U}})\to C^{\infty}((-\epsilon, \epsilon),\Omega^*(N, (E\oplus E)|_N)),$$
which extends to an isometry on the $L^2-$completions of the spaces due to the product structure of the metrics. The isometric transformation preserves the spectral properties of the transformed operators. Hence we can equivalently deal with the even part $\B_{\textup{even}}$ of the odd-signature operator associated to the Hilbert complex $(\domm, \DD)$ over $M$, under the isometric transformation $\Phi$. 
\\[3mm] The assumption of temporal gauge for the connection $\D$ implies with the same calculations as in [KL, Section 8.1]:  
\begin{align}\label{produkt}
\Phi \circ \B_{\textup{even}}\circ \Phi^{-1}=\gamma \left[\frac{d}{dx}+A\right],
\end{align}
where the operators $\gamma$ and $A$ are of the following form (compare [KL, Section 8.1]):
\begin{align*}
\gamma (\beta)=\left\{\begin{array}{ll}i^r(-1)^{p-1}\widetilde{*}_N\beta, & \textup{if} \ \beta \in \Omega^{2p}(N, (E\oplus E)|_N), \\
 i^r(-1)^{r-p-1}\widetilde{*}_N\beta, & \textup{if} \ \beta \in \Omega^{2p+1}(N, (E\oplus E)|_N).\end{array}\right. \\
 A (\beta)=\left\{\begin{array}{ll}-(\DD_N\widetilde{*}_N+\widetilde{*}_N\DD_N)\beta, & \textup{if} \ \beta \in \Omega^{2p}(N, (E\oplus E)|_N), \\
(\DD_N\widetilde{*}_N+\widetilde{*}_N\DD_N)\beta, & \textup{if} \ \beta \in \Omega^{2p+1}(N, (E\oplus E)|_N).\end{array}\right.
\end{align*}
Here $\DD_N= \D_N \oplus \D_N$ where $\D_N$ is the flat connection on $E|_N$ whose pullback to $E|_{\mathcal{U}}$ gives $\D|_{\mathcal{U}}$. Further $\widetilde{*}_N$ is the Hodge-star operator on $N$ acting anti-diagonally on $(E\oplus E)|_N$. We write 
\begin{align*}
\DD_N=\left(\begin{array}{cc}\D_N & 0 \\ 0 & \D_N \end{array}\right), \ \widetilde{*}_N=  \left(\begin{array}{cc}0 & *_N \\ *_N & 0 \end{array}\right).
\end{align*}                                                                          
\\[3mm] Consider next the odd signature operators $\B^j_{\textup{even}},j=1,2$ viewed as boundary value problems for $\B_{\textup{even}}|_{M_j}, j=1,2$ where the boundary conditions are to be identified. To visualize the structure involved, we distinguish notationally each direct sum component in $E \oplus E$: $$E\oplus E\equiv E^+\oplus E^-.$$ Decompose now $\Omega^*(N,(E\oplus E)|_N)$ as follows:
\begin{align}\label{decomp}
&\Omega^*(N, (E\oplus E)|_N)= \\ &\left[\Omega^{\textup{even}}(N,E^+|_N)\oplus \Omega^{\textup{odd}}(N,E^+|_N)\right] \oplus \left[\Omega^{\textup{even}}(N,E^-|_N)\oplus \Omega^{\textup{odd}}(N,E^-|_N)\right].\nonumber
\end{align}
The restriction map $R$ acts with respect to this decomposition as follows:
\begin{align}\label{restriction-decomposition}
R(\beta^+\oplus \beta^-)=[\iota^*(\beta^+)\oplus \iota^*(*\beta^-)]\oplus[\iota^*(\beta^-)\oplus \iota^*(*\beta^+)], \\
\textup{where} \ \beta^+\oplus \beta^-\in \Omega^{\textup{even}}(M,E^+\oplus E^-).
\nonumber 
\end{align}
Furthermore with respect to this decomposition operators $\gamma$ and $A$ are given by the following matrix form:
\begin{align}\label{ya}
\gamma&=\left(\begin{array}{cc}0 & \overline{\gamma} \\ \overline{\gamma} & 0 \end{array}\right), \ A=\left(\begin{array}{cc}0 & \overline{A} \\ \overline{A} & 0 \end{array}\right), \\ \nonumber
\overline{\gamma} (\beta)&=\left\{\begin{array}{ll}i^r(-1)^{p-1}*_N\beta, & \textup{if} \ \beta \in \Omega^{2p}(N, E|_N), \\
 i^r(-1)^{r-p-1}*_N\beta, & \textup{if} \ \beta \in \Omega^{2p+1}(N, E|_N).\end{array}\right. \\ \nonumber
 \overline{A} (\beta)&=\left\{\begin{array}{ll}-(\D_N*_N+*_N\D_N)\beta, & \textup{if} \ \beta \in \Omega^{2p}(N, E|_N), \\
(\D_N*_N+*_N\D_N)\beta, & \textup{if} \ \beta \in \Omega^{2p+1}(N, E|_N).\end{array}\right.
\end{align}
Note that $\gamma$ and $A$ satisfy the conditions \eqref{comm-gamma-A}. Recall now from [BV3, Lemma 3.6]
$$\dom (\B^j)=\dom \left(D^{GB}_{j, \textup{rel}}\oplus D^{GB}_{j, \textup{abs}}\right),$$ where $D^{GB}_j$ is the Gauss-Bonnet operator on $M_j$ associated to the connection $\D_j$. Hence the boundary conditions for $\B^j_{\textup{even}}$ are given as follows (see [BL1, Theorem 4.1], where the arguments are performed in the untwisted setup, but transfer analogously to the twisted case, provided product metric structures and a flat connection in temporal gauge)
\begin{align*}
&\beta =\beta^+\oplus \beta^- \in \dom (\B^j_{\textup{even}})\cap \Omega^{\textup{even}}(M_j, E^+\oplus E^-), \\
\textup{hence} \quad &\beta^+ \in \dom (D^{GB}_{\textup{rel}}), \quad \beta^- \in \dom (D^{GB}_{\textup{abs}}), \\
\textup{hence} \quad &\iota_j^*(\beta^+)=0, \quad \iota_j^*(*\beta^-)=0.
\end{align*}
According to \eqref{restriction-decomposition} we obtain under the isometric identification $\Phi$ over $U\cap M_j$ and with respect to the decomposition \eqref{decomp} the following matrix form for the boundary operators of $\B^j_{\textup{even}}$
\begin{align}
P=\left(\begin{array}{cccc} 1&0&0&0 \\0&1&0&0 \\0&0&0&0 \\ 0&0&0&0 \end{array}\right).
\end{align}
Note that $(I-P)$ again provides boundary conditions for $\B^j_{\textup{even}}$ with the components $E^+, E^-$ interchanged. The boundary operator $P$ obviously extends to a pseudo-differential operator of order zero. One checks explicitly by matrix calculations that $P$ satisfies the conditions \eqref{boundary2} and \eqref{boundary3}. The condition \eqref{boundary4} remains to be verified. 
\\[3mm] Being elliptic and self-adjoint, $A$ has discrete real spectrum with finite multiplicities. Discreteness of $A$ together with self-adjointness of $\B^j_{\textup{even}}$ implies with [BL3, Corollary 4.6 and Theorem 5.6] that $P$ indeed satisfies \eqref{boundary4}.
\\[3mm] Thus the conditions for the application of \eqref{splitting-eta-general} are satisfied and we obtain
\begin{align}\label{splitting-eta}
\eta(\B_{\textup{even}})=\eta(\B^1_{\textup{even}})+\eta(\B^2_{\textup{even}})-\tau_{\mu}(I-P_1,P,P_1),
\end{align}
where the operators $\B^1_{\textup{even}}, \B^2_{\textup{even}}$ denote the even parts of the odd-signature operators associated to the Hilbert complexes $(\domm_1,\DD_1), (\domm_2,\DD_2)$ respectively. Equivalently they are the self-adjoint realizations of the differential operators $\B_{\textup{even}}|_{M_1}, \B_{\textup{even}}|_{M_2}$ with the boundary conditions $P$ and $(I-P)$ respectively. Further $P_1$ denotes the Calderon projector of $\B^1_{\textup{even}}$ and $\tau_{\mu}$ the Maslov triple index.
\\[3mm] Due to self-adjointness of the odd-signature operator, the notion of (reduced) eta-invariant in [KL] for $\B_{\textup{even}}$ and $\B^j_{\textup{even}}, j=1,2$ coincides with the setup of [BV3, (4.1)].
\\[3mm] We obtain an analogous splitting formula in case of a trivial line bundle $M\times \C$ in the notation of [BV3, Proposition 4.6]
\begin{align}\label{splitting-eta-trivial}
\eta(\B_{\textup{trivial}})=\eta(\B^1_{\textup{trivial}})+\eta(\B^2_{\textup{trivial}})-\tau_{\mu}(I-P_{1, \textup{trivial}},P_{\textup{trivial}},P_{1, \textup{trivial}}),
\end{align}
with the obvious notation. This formula will be necessary in order to obtain a splitting for the metric-anomaly term in refined analytic torsion.
\\[3mm] In other words, the phase of refined analytic torsion is given in part by the rho invariant of $\B_{\textup{even}}$, which is defined (cf. [KL, Definition 8.17]) as the eta-invariant of the operator minus the metric anomaly correction term. The results \eqref{splitting-eta} and \eqref{splitting-eta-trivial} give together a splitting formula for the rho-invariant, which constitutes to the complex phase of refined analytic torsion.

\section{Poincare Duality for manifolds with boundary}\label{PD-boundary}
We continue in the setup and the notation, fixed in Section \ref{gluing-statement}. Denote by $\triangle_{j, \textup{rel/abs}}$ the Laplacians of the Hilbert complexes $(\dom_{j, \textup{min/max}},\D_{j, \textup{min/max}})$ respectively. The coefficient $j$ refers to the base manifold $M_j, j=1,2$. Consider the Hodge star operator $*$ on $M$ and the associated chirality operator 
$$\G :=i^r(-1)^{\frac{k(k+1)}{2}}*:\Omega^k(M,E)\to \Omega^{m-k}(M,E), \quad r:= (\dim M+1)/2.$$
We do not indicate the restriction of $\G$ to $M_j, j=1,2$ by a subscript $j$, since it will always be clear from the action. By the properties of the chirality operator $\G$ in [BV3, Lemma 3.3], we infer:
\begin{align}\label{PD-laplace}
\triangle_{j, \textup{abs/rel}}=\G \circ \triangle_{j, \textup{rel/abs}}\circ \G, \\
\textup{and hence } \  \G : \ker \triangle_{j, \textup{rel/abs}} \ \widetilde{\rightarrow} \ \ker \triangle_{j, \textup{abs/rel}}.
\end{align}
Note the Hodge isomorphisms
\begin{align*}
\ker \triangle_{j, \textup{rel/abs}} \ \widetilde{\rightarrow} \ H^*_{\textup{rel/abs}}(M_j,E), \quad \phi \ \longmapsto \ [\phi],
\end{align*}
where $H^*_{\textup{rel/abs}}(M_j,E)$ denote the de Rham cohomologies of the Fredholm complexes $(\dom_{j, \textup{min/max}}, \D_{j, \textup{min/max}}), j=1,2$ respectively. Hence the chirality operator induces under the Hodge isomorphisms the so-called Poincare duality on manifolds with boundary:
 $$\G : H^k_{\textup{rel/abs}}(M_j,E) \ \widetilde{\rightarrow} \ H^{m-k}_{\textup{abs/rel}}(M_j,E), \quad k=0,..,m=\dim M.$$
Next we introduce two pairings. Let $\w_1=s_1\otimes \chi_1$ and $\w_2=s_2\otimes \chi_2$ be two differential forms in $\Omega^*(M_j,E)$ with $s_1,s_2 \in C^{\infty}(M_j,E)$ and $\chi_1, \chi_2\in \Omega^*(M_j)$. Put: $$h^E(\w_1\wedge \w_2):=h^E(s_1,s_2)\cdot \chi_1 \wedge \chi_2.$$ This action extends by linearity to arbitrary differential forms in $\Omega^*(M_j,E)$. With this notation we define a pairing, which is the Hilbert structure on $H^*_{\textup{rel/abs}}(M_j,E)$ induced by the $L^2-$structure on the harmonic forms:
\begin{align}
\langle \cdot \, , \cdot \rangle_{L^2_j}:H^k_{\textup{rel/abs}}(M_j,E)&\times H^k_{\textup{rel/abs}}(M_j,E) \ \rightarrow \ \C \nonumber \\
[\w]&, [\eta] \ \longmapsto \ \int_{M_j}h^E(\w \wedge *\eta),
\end{align}
where $\w$ and $\eta$ are the harmonic representatives of $[\w]$ and $[\eta]$ respectively. The second pairing is the Poincare duality on Riemannian manifolds with boundary and is in fact independent of a choice of representatives:
\begin{align}
\langle \cdot \, , \cdot \rangle_{P_j}:H^k_{\textup{rel/abs}}(M_j,E)&\times H^{m-k}_{\textup{abs/rel}}(M_j,E) \ \rightarrow \ \C \nonumber \\
[\w]&, [\eta] \ \longmapsto \ \int_{M_j}h^E(\w \wedge \eta).
\end{align}
Both pairings are non-degenerate and induce canonical identifications between cohomology and its dual:
\begin{align*}
&\#_{L^2_j}:H^k_{\textup{rel/abs}}(M_j,E)\widetilde{\rightarrow} \left(H^k_{\textup{rel/abs}}(M_j,E)\right)^*, \quad [\w] \ \longmapsto \langle \cdot \, , \ [\w]\rangle_{L^2_j}, \\
&\#_{P_j}:H^k_{\textup{rel/abs}}(M_j,E)\widetilde{\rightarrow} \left(H^{m-k}_{\textup{abs/rel}}(M_j,E)\right)^*, \quad [\w] \ \longmapsto \langle \cdot \, , \ [\w]\rangle_{P_j}.
\end{align*}
Both maps are linear, with the Hermitian metric $h^E$ set to be linear in the second component. The next proposition puts the constructions above into relation:
\begin{prop}\label{PD-relation} The action of $\#_{L^2_j}\circ \G $ and $\#_{P_j}$ on $H^k_{\textup{rel/abs}}(M_j,E)$ satisfies
$$\#_{L^2_j}\circ \G = i^r (-1)^{\frac{k(k+1)}{2}}\#_{P_j}, \quad r=(\dim M +1)/2.$$
\end{prop}
\begin{proof}
Let $[\w]\in H^k_{\textup{rel/abs}}(M_j,E)$ and $[\eta]\in H^k_{\textup{abs/rel}}(M_j,E)$ with $\w, \eta$ being the harmonic representatives of $[\w],[\eta]$ respectively. Using $\G^2=\textbf{1}$ we get 
\begin{align*}
*\G\w = i^{-r}(-1)^{\frac{(m-k)(m-k+1)}{2}}\G\circ \G \w=
\\ = i^{-r}(-1)^{\frac{k(k+1)}{2}+\frac{m(m+1)}{2}}\w= i^r(-1)^{\frac{k(k+1)}{2}}\w.
\end{align*}
Due to linearity of the Hermitian metric in the second component we finally obtain:
\begin{align*}
(\#_{L^2_j}\circ \G)([\w])[\eta]=\langle [\eta],\G [\w]\rangle_{L^2_j}=\int_{M_j}h^E(\eta \wedge *\G \w)=\\= i^r (-1)^{\frac{k(k+1)}{2}}\int_{M_j}h^E(\eta \wedge \w)= i^r (-1)^{\frac{k(k+1)}{2}}\#_{P_j}([\w])[\eta].
\end{align*}
\end{proof}\ \\
\\[-7mm] A similar discussion works also on the closed Riemannian split manifold $M$. In particular we obtain as before the pairings $\langle \cdot \, , \cdot \rangle_{L^2}$ and $\langle \cdot \, , \cdot \rangle_P$ with the associated identifications $\#_{L^2}$ and $\#_{P}$ respectively, over the manifold $M$. As in Proposition \ref{PD-relation} we obtain for the action of $\#_{L^2}\circ \G $ and $\#_{P}$ on $H^k(M,E)$ the following relation
\begin{align}
\#_{L^2}\circ \G = i^r (-1)^{\frac{k(k+1)}{2}}\#_{P}, \quad r=(\dim M +1)/2.
\end{align} 
Next we consider a complex, that takes the splitting $M=M_1 \cup_N M_2$ into account. Let $\iota_{j}:N\hookrightarrow M_j, j=1,2$ be the natural inclusions. Put
$$\Omega^*(M_1\#M_2,E):=\{(\w_1,\w_2) \in \Omega^*(M_1,E)\oplus \Omega^*(M_2,E)|\iota^*_1\w_1=\iota^*_2\w_2\}.$$
Denote the restrictions of the flat connection $\D$ to $M_j, j=1,2$ by $\D_j$, and extend the restrictions by Leibniz rule to operators on the complexes $\Omega^*(M_j,E),j=1,2$. We put further
$$\D_S(\w_1,\w_2):=(\D_1\w_1,\D_2\w_2).$$
This operation respects the transmission condition of $\Omega^*(M_1\#M_2,E)$ and further its square is obviously zero. Therefore $\D_S$ turns the graded vector space $\Omega^*(M_1\#M_2,E)$ into a complex, denoted by
\begin{align}\label{split-complex}
(\Omega^*(M_1\#M_2,E), \D_S).
\end{align}
The natural $L^2-$structure on $\Omega^*(M_1,E)\oplus \Omega^*(M_2,E)$, induced by the metrics $g^M$ and $h^E$ is defined on any $\w=(\w_1,\w_2), \eta=(\eta_1,\eta_2)$ as follows 
\begin{align}
\langle \w , \eta \rangle_{L^2}:=\sum\limits_{j=1}^2\langle \w_j, \eta_j\rangle_{L^2_j}.
\end{align}
In order to analyze the associated Laplace operators, consider first the adjoint to $\D_S$ operator $\D^*_S$ in $\Omega^*(M_1,E)\oplus \Omega^*(M_2,E)$ with domain of definition $\dom (\D^*_S)$ consisting of elements $\w=(\w_1,\w_2) \in \Omega^*(M_1,E)\oplus \Omega^*(M_2,E)$ such that the respective linear functionals on any $\eta =(\eta_1,\eta_2) \in \Omega^*(M_1\#M_2,E)$
\begin{align*}
L_{\w}(\eta)=\langle \w , \D_S\eta\rangle_{L^2}
\end{align*}
are continuous in $\Omega^*(M_1\#M_2,E)$ with respect to the natural $L^2-$norm of $\eta$. As a consequence of Stokes' formula we find for such elements $\w \in \dom (\D^*_S)$ that the following transmission condition has to hold 
\begin{align}\label{split-gamma}
*\w=(*\w_1,*\w_2)\in \Omega^*(M_1\#M_2,E),
\end{align}
where $*$ also denotes the restrictions of the usual Hodge star operator on $M$ to $M_1$ and $M_2$. The Laplace operator $\triangle_S=\D_S^*\D_S+\D_S\D_S^*$ of the complex \eqref{split-complex} acts on the obvious domain of definition
\begin{align}\label{dom-def}
\dom (\triangle_S)=\{&\w \in \Omega^*(M_1\#M_2,E)| \\ &\w \in \dom (\D^*_S), \D_S\w \in \dom (\D^*_S), \D^*_S\w \in \Omega^*(M_1\#M_2,E)\}. \nonumber 
\end{align}
The $Dom (\triangle_S)$ is defined as the completion of $\dom(\triangle_S)$ with respect to the graph topology norm. The Laplacian $\triangle_S$ with domain $Dom (\triangle_S)$ is a self-adjoint operator in the $L^2-$completion of $\Omega^*(M_1,E)\oplus \Omega^*(M_2,E)$. 
\\[3mm] For the spectrum of $\triangle_S$ we refer to the theorem below, established essentially by S. Vishik in [V, Proposition 1.1] in the untwisted setup.
\begin{thm}\label{split-laplacian}
The generalized eigenforms of the Laplacian $\triangle_S$ and the generalized eigenforms of the Laplacian $\triangle$ associated to the twisted de Rham complex $(\Omega^*(M,E),\D)$ coincide.
\end{thm}
\begin{proof}
The conditions for $(\w_1,\w_2)\in \dom (\triangle_S)$ translate with \eqref{split-gamma} equivalently to
\begin{align}
&\iota^*_1\w_1=\iota^*_2\w_2, \quad \iota^*_1(*\w_1)=\iota^*_2(*\w_2),\nonumber \\
&\iota^*_1(*\D_1\w_1)=\iota^*_2(*\D_2\w_2), \quad \iota^*_1(\D^t_1\w_1)=\iota^*_2(\D^t_2\w_2).\label{split-conditions}
\end{align}
Any eigenform $\w$ of $\triangle$ is smooth and thus $(\w|_{M_1},\w|_{M_2})$ satisfies the conditions \eqref{split-conditions}. Thus any eigenform $\w\equiv (\w|_{M_1},\w|_{M_2})$ of $\triangle$ belongs to $\dom (\triangle_S)$ and hence is an eigenform of $\triangle_S$. We need to show the converse statement.
\\[3mm] Let $(\w_1,\w_2)\in \dom (\triangle_S)$ be an eigenform of $\triangle_S$. Then for any $k \in \N$ the element $(\triangle_1^k\w_1, \triangle_2^k\w_2)$ satisfies the conditions \eqref{split-conditions}. Fix local coordinates $(x,y)$ in the collar neighborhood $(-\epsilon, \epsilon) \times N$ of $N \subset M$ with $x\in (-\epsilon, \epsilon)$ being the normal coordinate and $y \in N$ the local coordinates on $N$. Then the conditions \eqref{split-conditions} imply for $k=1$ $$\frac{\partial \w_1(x=0,y)}{\partial x}=\frac{\partial \w_2(x=0,y)}{\partial x}.$$ Iterative application of the conditions \eqref{split-conditions} to $(\triangle_1^k\w_1, \triangle_2^k\w_2)$ for $k \in \N$ shows
\begin{align}\label{smoothness}
\forall k \in \N: \ \frac{\partial^k \w_1(x=0,y)}{\partial x^k}=\frac{\partial^k \w_2(x=0,y)}{\partial x^k}.
\end{align}
The eigenform $(\w_1,\w_2)$ consists of smooth eigenforms $\w_j$ over $M_j,j=1,2$. The result \eqref{smoothness} shows smoothness on $N\subset M$. Thus $(\w_1,\w_2)$ can be viewed as a smooth differential form over $M$ and so lies in $\dom (\triangle)$ and hence is an eigenform of $\triangle$ as well. This proves the theorem.
\end{proof}
\begin{cor}\label{split-cohomology}
The Laplacian $\triangle_S$ on $Dom(\triangle_S)$ is a Fredholm operator and 
\begin{align*}
H^*(M_1\#M_2,E):=H^*(\Omega^*(M_1\#M_2,E),\D_S)\cong H^*_{\textup{dR}}(M,E).
\end{align*}
\end{cor}\ \\
\\[-7mm] The corollary is an obvious consequence of Theorem \ref{split-laplacian} and the Hodge-isomorphism. Therefore the pairings $\langle \cdot \, , \cdot \rangle_{L^2}$ and $\langle \cdot \, , \cdot \rangle$ with the associated identifications $\#_{L^2}$ and $\#_{P}$ respectively, over the manifold $M$ give rise to pairings and maps on $H^*(M_1\#M_2,E)$. We do not introduce a distinguished notation for these induced constructions
\begin{align}
\langle \cdot \, , \cdot \rangle_{L^2}:H^k(M_1\#M_2,E)&\times H^k(M_1\#M_2,E) \ \rightarrow \ \C \nonumber \\
[(\w_1,\w_2)]&, [(\eta_1,\eta_2)] \ \longmapsto \ \sum\limits_{j=1}^2\int_{M_j}h^E(\w_j \wedge *\eta_j),\\
\langle \cdot \, , \cdot \rangle_{P}:H^k(M_1\#M_2,E)&\times H^{m-k}(M_1\#M_2,E) \ \rightarrow \ \C \nonumber \\
[(\w_1,\w_2)]&, [(\eta_1,\eta_2)] \ \longmapsto \ \sum\limits_{j=1}^2\int_{M_j}h^E(\w_j \wedge \eta_j),
\end{align}
where $(\w_1,\w_2), (\eta_1,\eta_2)$ are a priori harmonic representatives of the corresponding cohomology classes, due to the Hodge isomorphisms applied in the identification of Corollary \ref{split-cohomology}. A posteriori we find by the next lemma that the pairing $\langle \cdot \, , \cdot \rangle_{P}$ like the pairings $\langle \cdot \, , \cdot \rangle_{P_j},j=1,2$ is well-defined on cohomology classes, i.e. need not be evaluated on harmonic representatives only.
\begin{lemma}\label{harm-rep}
The pairing $\langle \cdot , \cdot \rangle_{P}$ is a well-defined pairing on cohomology.
\end{lemma}
\begin{proof}
Let $[(\w_1,\w_2)]\in H^k(M_1\#M_2,E)$ be a cohomology class with a representative $(\w_1,\w_2)+\D_S(\gamma_1,\gamma_2)$ where $(\w_1,\w_2)\in \ker \D_S$ and $(\w_1,\w_2), (\gamma_1,\gamma_2) \in \Omega^*(M_1\#M_2,E)$, in particular $$\iota^*_1\w_1=\iota^*_2\w_2, \ \iota^*_1\gamma_1=\iota^*_2\gamma_2.$$
Similarly let $[(\eta_1,\eta_2)]\in H^{m-k}(M_1\#M_2,E)$. Choose a representative $(\eta_1,\eta_2)+\D_S(\xi_1,\xi_2)$ with $(\eta_1,\eta_2)\in \ker \D_S$ and $(\eta_1,\eta_2), (\xi_1,\xi_2) \in \Omega^*(M_1\#M_2,E)$. We compute
\begin{align}
\sum\limits_{j=1}^2\int_{M_j}h^E((\w_j+\D_j\gamma_j)\wedge (\eta_j + \D_j\xi_j))-\sum\limits_{j=1}^2\int_{M_j}h^E(\w_j\wedge \eta_j)= \nonumber\\
= \sum\limits_{j=1}^2\int_{M_j}h^E(\D_j\gamma_j\wedge \eta_j )+ \sum\limits_{j=1}^2\int_{M_j}h^E(\w_j\wedge \D_j\xi_j) + \nonumber \\ + \sum\limits_{j=1}^2\int_{M_j}h^E(\D_j\gamma_j\wedge \D_j\xi_j).\label{summands}
\end{align}
In order to verify that the pairing $\langle \cdot \, , \cdot \rangle_P$ is a well-defined pairing on cohomology we need to show that the last three summands in \eqref{summands} are zero. Consider the first summand, the other two are dealt with analogously. Under the assumption of flatness of $\D$ we get
\begin{align*}
d h^E(\gamma_j\wedge \eta_j)=h^E(\D_j\gamma_j \wedge \eta)+(-1)^{k-1}h^E(\gamma_j \wedge \D_j\eta)=h^E(\D_j\gamma_j \wedge \eta) \\
\Rightarrow \sum\limits_{j=1}^2\int_{M_j}h^E(\D_j\gamma_j\wedge \eta_j )= \sum\limits_{j=1}^2\int_{M_j}d h^E(\gamma_j\wedge \eta_j)=\\=\sum\limits_{j=1}^2\int_{\partial M_j}\iota^*_j h^E(\gamma_j\wedge \eta_j).
\end{align*}
Since $\iota^*_1\gamma_1=\iota^*_2\gamma_2$ and $\iota^*_1\eta_1=\iota^*_2\eta_2$ we find  $$\iota^*_1 h^E(\gamma_1\wedge \eta_1)=\iota^*_2 h^E(\gamma_2\wedge \eta_2).$$ However the orientations on $N=\partial M_1=\partial M_2$ induced from $M_1$ and $M_2$ are opposite. Hence the two integrals over $M_1$ and $M_2$ cancel. This completes the argumentation.
\end{proof}

\section{Commutative diagramms in cohomological algebra}\label{cohom-algebra} 
Consider the short exact sequences of complexes:
\begin{align*}
0 \rightarrow (\Omega^*_{\min}(M_1,E),\D_1)\xrightarrow{\A} (\Omega^*(M_1\#M_2,E),\D_S) \xrightarrow{\beta} (\Omega^*_{\max}(M_2,E),\D_2)\rightarrow 0, \\
0 \rightarrow (\Omega^*_{\min}(M_2,E),\D_2)\xrightarrow{\A'} (\Omega^*(M_1\#M_2,E),\D_S) \xrightarrow{\beta'} (\Omega^*_{\max}(M_1,E),\D_1)\rightarrow 0,
\end{align*}
where $\A(\w) =(\w,0), \A'(\w)=(0,\w)$ and $\beta(\w_1,\w_2)=\w_2, \beta'(\w_1,\w_2)=\w_1$. The exactness at the first and the second complex of both sequences is clear by construction. The surjectivity of $\beta$ and $\beta'$ is clear, since $\Omega^*_{\max}(M_j,E),j=1,2$ consists of smooth differential forms over $M_j$ which are in particular smooth at the boundary. These short exact sequences of complexes induce long exact sequences on cohomology:
\begin{align}\nonumber &\mathcal{H}: \ 
...H^k_{\textup{rel}}(M_1,E)\xrightarrow{\A^*}H^k(M_1\#M_2,E)\xrightarrow{\beta^*}H^k_{\textup{abs}}(M_2,E)\xrightarrow{\delta^*}H^{k+1}_{\textup{rel}}(M_1,E)... \\ \nonumber &\mathcal{H'}\!: ..H^k_{\textup{rel}}(M_2,E)\xrightarrow{\A'^*}H^k(M_1\#M_2,E)\xrightarrow{\beta'^*}H^k_{\textup{abs}}(M_1,E)\xrightarrow{\delta'^*}H^{k+1}_{\textup{rel}}(M_2,E)\\
\label{LES-H}
\end{align}
The first long exact sequence is related to the dual of the second long exact sequence by the diagramm below, where $\A'_*,\beta'_*,\delta'_*$ denote the dualizations of $\A'^*,\beta'^*,\delta'^*$ respectively.
\small{
\begin{align}\label{comm-diagramm}
\begin{array}{ccccccc}
H^k_{\textup{rel}}(M_1,E)\!\!&\!\! \xrightarrow{\A^*} \!\!&\!\! H^k(M_1\#M_2,E) \!\!&\!\! \xrightarrow{\beta^*} \!\!&\!\! H^k_{\textup{abs}}(M_2,E) \!\!&\!\!  \xrightarrow{\delta^*} \!\!&\!\! H^{k+1}_{\textup{rel}}(M_1,E) \\ \!\!&\!\! \!\!&\!\! \!\!&\!\! \!\!&\!\! \!\!&\!\! \!\!&\!\! \\
 \#_{L^2_1}\circ \G \downarrow \!\!&\!\!\!\!&\!\! \#_{L^2}\circ \G \downarrow \!\!&\!\!\!\!&\!\! \#_{L^2_2}\circ \G \downarrow  \!\!&\!\!  \!\!&\!\! \#_{L^2_1}\circ \G \downarrow \\ \!\!&\!\! \!\!&\!\! \!\!&\!\! \!\!&\!\! \!\!&\!\! \!\!&\!\! \\
H^{m-k}_{\textup{abs}}(M_1,E)^* \!\!\!&\!\! \xrightarrow{\beta'_*} \!\!\!&\!\! H^{m-k}(M_1\#M_2,E)^* \!\!\!&\!\! \xrightarrow{\A'_*} \!\!\!&\!\! H^{m-k}_{\textup{rel}}(M_2,E)^* \!\!\!&\!\! \xrightarrow{\delta'_*} \!\!\!&\!\! H^{m-k-1}_{\textup{abs}}(M_1,E)^* 
\end{array} 
\end{align}}
\normalsize

\begin{thm}\label{comm-diagramm2}
The diagramm \eqref{comm-diagramm} is commutative.
\end{thm}
\begin{proof}
We need to verify commutativity of three types of squares in the diagramm. Consider the first type of squares:
\begin{align*}
\begin{xy}
  \xymatrix{
      H^k_{\textup{rel}}(M_1,E) \ar[r]^{\A^*} \ar[d]_{\#_{L^2_1}\circ \G}    &   H^k(M_1\#M_2,E) \ar[d]^{\#_{L^2}\circ \G}    \\
      H^{m-k}_{\textup{abs}}(M_1,E)^* \ar[r]_{\beta'_*}                     &   H^{m-k}(M_1\#M_2,E)^*.
  }
\end{xy}
\end{align*}
Let $[\w]\in H^k_{\textup{rel}}(M_1,E)$. Recall
\begin{align*}
&\#_{L^2_1}\circ \G=i^r(-1)^{\frac{k(k+1)}{2}}\#_{P_1} \ \textup{on} \ H^k_{\textup{rel}}(M_1,E), \\
&\#_{L^2}\circ \G=i^r(-1)^{\frac{k(k+1)}{2}}\#_{P} \ \textup{on} \ H^k(M_1\#M_2,E).
\end{align*}
The maps $\#_{P_1},\#_P$ are well-defined identifications on cohomology, due to Lemma \ref{harm-rep}. Let $[(\eta_1,\eta_2)] \in H^{m-k}(M_1\#M_2,E)$ and compute:
\begin{align*}
&(\beta'_*\circ \#_{L^2_1}\circ \G)[\w]([\eta_1,\eta_2])-(\#_{L^2}\circ \G \circ \A^*)[\w]([\eta_1,\eta_2])= \\
=& i^r(-1)^{\frac{k(k+1)}{2}} \left\{ \langle \beta'(\eta_1,\eta_2), \w \rangle_{P_1}- \langle (\eta_1,\eta_2), \A\w \rangle_P \right\}= \\
=& i^r(-1)^{\frac{k(k+1)}{2}} \left\{ \int_{M_1}h^E(\eta_1\wedge \w) - \int_{M_1}h^E(\eta_1\wedge \w) \right\}=0.
\end{align*}
Consider now the second type of squares in the diagramm \eqref{comm-diagramm}.
\begin{align*}
\begin{xy}
  \xymatrix{
      H^k(M_1\#M_2,E) \ar[r]^{\beta^*} \ar[d]^{\#_{L^2}\circ \G}   &  H^k_{\textup{abs}}(M_2,E)  \ar[d]^{\#_{L^2_2}\circ \G}   \\
      H^{m-k}(M_1\#M_2,E)^* \ar[r]^{\A'_*}                        & H^{m-k}_{\textup{rel}}(M_2,E)^* . 
  }
\end{xy}
\end{align*}
Let $[(\w_1,\w_2)]\in H^k(M_1\#M_2,E)$ and $[\eta]\in H^{m-k}_{\textup{rel}}(M_2,E)$. As before the maps in the diagramm are independent of particular choices of representatives, so we compute:
\begin{align*}
(\A'_*\circ \#_{L^2}\circ \G)[(\w_1,\w_2)]([\eta])-(\#_{L^2_2}\circ \G \circ \beta^*)[(\w_1,\w_2)]([\eta])= \\
= i^r(-1)^{\frac{k(k+1)}{2}} \left\{ \langle \A'\eta , (\w_1,\w_2)\rangle_P - \langle \eta,  \beta(\w_1,\w_2)\rangle_{P_2}\right\}= \\
= i^r(-1)^{\frac{k(k+1)}{2}} \left\{\int_{M_2}h^E(\eta\wedge \w_2) - \int_{M_2}h^E(\eta\wedge \w_2) \right\}=0.
\end{align*}
Consider finally the third type of squares. 
\begin{align}\label{third-type}
\begin{xy}
  \xymatrix{
      H^k_{\textup{abs}}(M_2,E)  \ar[r]^{\delta^*} \ar[d]^{\#_{L^2_2}\circ \G}  & H^{k+1}_{\textup{rel}}(M_1,E) \ar[d]^{\#_{L^2_1}\circ \G}  \\
      H^{m-k}_{\textup{rel}}(M_2,E)^* \ar[r]^{\delta'_*} &  H^{m-k-1}_{\textup{abs}}(M_1,E)^*. 
  }
\end{xy}
\end{align}
To prove commutativity of this diagramm, we need a precise understanding of the connecting homomorphisms $\delta^*, \delta'^*$. Note for this the following diagramm of short exact sequences of complexes: 

\begin{align}\label{ses-diagramm1}
\begin{array}{ccccccccc}
0 \!\!&\!\! \rightarrow \!\!&\!\! (\Omega^*_{\min}(M_1,E),\D_1) \!\!&\!\! \xrightarrow{\A} \!\!&\!\! (\Omega^*(M_1\#M_2,E),\D_S) \!\!&\!\! \xrightarrow{\beta} \!\!&\!\! (\Omega^*_{\max}(M_2,E),\D_2) \!\!&\!\! \rightarrow  \!\!&\!\! 0 \\ 
  \!\!&\!\!             \!\!&\!\!                       \!\!&\!\!                  \!\!&\!\!                     \!\!&\!\!                     \!\!&\!\!          \!\!&\!\!               \!\!&\!\! \\
  \!\!&\!\!             \!\!&\!\!  \|                         \!\!&\!\!                  \!\!&\!\!       \|                    \!\!&\!\!                     \!\!&\!\!    \beta_{\pi} \uparrow      \!\!&\!\!               \!\!&\!\! \\
    \!\!&\!\!             \!\!&\!\!                       \!\!&\!\!                  \!\!&\!\!                     \!\!&\!\!                     \!\!&\!\!          \!\!&\!\!               \!\!&\!\! \\
0 \!\!&\!\! \rightarrow \!\!&\!\! (\Omega^*_{\min}(M_1,E),\D_1) \!\!&\!\! \xrightarrow{\A} \!\!&\!\! (\Omega^*(M_1\#M_2,E),\D_S) \!\!&\!\! \xrightarrow{\pi}   \!\!&\!\! \overline{(\Omega^*_{\max}(M_2,E),\D_2)} \!\!&\!\! \rightarrow  \!\!&\!\! 0.
\end{array}
\end{align}
The complex $\overline{(\Omega^*_{\max}(M_2,E),\D_2)}$ in the lower short exact sequence is the natural quotient of complexes
$$\overline{(\Omega^*_{\max}(M_2,E),\D_2)}:= \frac{(\Omega^*(M_1\#M_2,E),\D_S)}{\A (\Omega^*_{\min}(M_1,E),\D_1)}.$$ The complex map $\pi$ is the natural projection. The map $\beta_{\pi}$ is an isomorphism of complexes:
\begin{align*}
\beta_{\pi}: \overline{(\Omega^*_{\max}(M_2,E),\D_2)} &\rightarrow (\Omega^*_{\max}(M_2,E),\D_2) \\
[(\w_1,\w_2)] & \mapsto \beta(\w_1,\w_2)=\w_2.
\end{align*}
The diagramm \eqref{ses-diagramm1} of short exact sequences of complexes obviously commutes. Hence the associated diagramm of long exact sequences on cohomology is also commutative and in particular we obtain the following commutative diagramm: 
\begin{align}\label{comm1}
\begin{xy}
  \xymatrix{
      H^k_{\textup{abs}}(M_2,E)  \ar[r]^{\delta^*}  & H^{k+1}_{\textup{rel}}(M_1,E) \ar@{=}[d]  \\
      H^{k}(\overline{(\Omega^*_{\max}(M_2,E),\D_2)}) \ar[r]^/.7em/{d^*} \ar[u]^{\beta_{\pi}^*} &  H^{k+1}_{\textup{rel}}(M_1,E). 
  }
\end{xy}
\end{align}
The vertical map $\beta_{\pi}^*$ is the isomorphism induced by $\beta_{\pi}$ and $\delta^*, d^*$ are the connecting homomorphisms of the long exact sequences associated to the lower and upper short exact sequence of complexes of \eqref{ses-diagramm1}, respectively.
\\[3mm] The connecting homorphism $d^*$ is easily defined. Let namely $[(\w_1,\w_2)]\in H^k(\overline{(\Omega^*_{\max}(M_2,E),\D_2)})$. Any of its representatives $(\w_1,\w_2)\in \Omega^k(M_1\#M_2,E)$ satisfies $\D_S(\w_1,\w_2)=(\D\w_1,0)\in \A(\Omega^*_{\min}(M_1,E),\D_1)$ by definition. Then  $$d^*[(\w_1,\w_2)]=[\D_1\w_1]\in H^{k+1}_{\textup{rel}}(M_1,E).$$ 
Consider now the next diagramm of short exact sequences of complexes: 
\begin{align}\label{ses-diagramm2}
\begin{array}{ccccccccc}
0 \!\!&\!\! \rightarrow \!\!&\!\! (\Omega^*_{\min}(M_2,E),\D_2) \!\!&\!\! \xrightarrow{\A'} \!\!&\!\! (\Omega^*(M_1\#M_2,E),\D_S) \!\!&\!\! \xrightarrow{\beta'} \!\!&\!\! (\Omega^*_{\max}(M_1,E),\D_1) \!\!&\!\! \rightarrow  \!\!&\!\! 0 \\ 
  \!\!&\!\!             \!\!&\!\!                       \!\!&\!\!                  \!\!&\!\!                     \!\!&\!\!                     \!\!&\!\!          \!\!&\!\!               \!\!&\!\! \\
  \!\!&\!\!             \!\!&\!\!  \|                         \!\!&\!\!                  \!\!&\!\!       \|                    \!\!&\!\!                     \!\!&\!\!    \beta'_{\pi} \uparrow      \!\!&\!\!               \!\!&\!\! \\
    \!\!&\!\!             \!\!&\!\!                       \!\!&\!\!                  \!\!&\!\!                     \!\!&\!\!                     \!\!&\!\!          \!\!&\!\!               \!\!&\!\! \\
0 \!\!&\!\! \rightarrow \!\!&\!\! (\Omega^*_{\min}(M_2,E),\D_2) \!\!&\!\! \xrightarrow{\A'} \!\!&\!\! (\Omega^*(M_1\#M_2,E),\D_S) \!\!&\!\! \xrightarrow{\pi'}   \!\!&\!\! \overline{(\Omega^*_{\max}(M_1,E),\D_1)} \!\!&\!\! \rightarrow  \!\!&\!\! 0.
\end{array}
\end{align}
The complex $\overline{(\Omega^*_{\max}(M_1,E),\D_1)}$ in the lower short exact sequence is the natural quotient of complexes
$$\overline{(\Omega^*_{\max}(M_1,E),\D_1)}:= \frac{(\Omega^*(M_1\#M_2,E),\D_S)}{\A (\Omega^*_{\min}(M_2,E),\D_2)}.$$ The complex map $\pi'$ is the natural projection. The map $\beta'_{\pi}$ is an isomorphism of complexes:
\begin{align*}
\beta'_{\pi}: \overline{(\Omega^*_{\max}(M_1,E),\D_1)} &\rightarrow (\Omega^*_{\max}(M_1,E),\D_1) \\
[(\w_1,\w_2)] & \mapsto \beta'(\w_1,\w_2)=\w_1.
\end{align*}
The diagramm \eqref{ses-diagramm2} of short exact sequences of complexes obviously commutes. Hence the associated diagramm of long exact sequences on cohomology is also commutative and in particular we obtain the following commutative diagramm: 
\begin{align}\label{comm2}
\begin{xy}
  \xymatrix{
      H^k_{\textup{abs}}(M_1,E)  \ar[r]^{\delta'^*}  & H^{k+1}_{\textup{rel}}(M_2,E) \ar@{=}[d]  \\
      H^{k}(\overline{(\Omega^*_{\max}(M_1,E),\D_1)}) \ar[r]^/.7em/{d'^*} \ar[u]^{\beta_{\pi}'^{*}} &  H^{k+1}_{\textup{rel}}(M_2,E). 
  }
\end{xy}
\end{align}
The vertical map $\beta_{\pi}'^*$ is the isomorphism induced by $\beta_{\pi}'$ and $\delta'^*, d'^*$ are the connecting homomorphisms of the long exact sequences associated to the lower and upper short exact sequence of complexes of \eqref{ses-diagramm2}, respectively.
\\[3mm] The connecting homomorphism $d'^*$ is easily defined. Let namely any $[(\eta_1,\eta_2)]\in H^{m-k-1}(\overline{(\Omega^*_{\max}(M_1,E),\D_1)})$. For any representative $(\eta_1,\eta_2)\in \Omega^{m-k-1}(M_1\#M_2,E)$ we have $$\D_S(\eta_1,\eta_2)=(0,\D_2\eta_2)\in \A'(\Omega^*_{\min}(M_2,E),\D_2)$$ by definition. We obtain for the connecting homomorphism $d'^*$ $$d'^*[(\eta_1,\eta_2)]=[\D_2\eta_2]\in H^{m-k}_{\textup{rel}}(M_2,E).$$
Now the pairings $\langle \cdot\, ,\cdot \rangle_{P_1}, \langle \cdot\, ,\cdot \rangle_{P_2}$, introduced in Section \ref{PD-boundary} induce via the isomorphisms on cohomology $\beta_{\pi}^*, \beta_{\pi}'^*$ the analogous pairings:
\begin{align*}
&\overline{\langle \cdot\, ,\cdot \rangle}_{P_2}:= \langle \cdot\, , \beta_{\pi}^*(\cdot) \rangle_{P_2}: H^{m-k}_{\textup{rel}}(M_2,E)\times H^{k}(\overline{(\Omega^*_{\max}(M_2,E),\D_2)}) \rightarrow \C, \\
&\overline{\langle \cdot\, ,\cdot \rangle}_{P_1}:= \langle \beta_{\pi}'^*(\cdot) \, , \cdot \rangle_{P_1}: H^{m-k-1}(\overline{(\Omega^*_{\max}(M_1,E),\D_1)})\times H^{k+1}_{\textup{rel}}(M_1,E) \rightarrow \C.
\end{align*}
These pairings induce the following identifications
\begin{align*}
\overline{\#}_{P_2}:H^{k}(\overline{(\Omega^*_{\max}(M_2,E),\D_2)})\widetilde{\rightarrow} \left(H^{m-k}_{\textup{rel}}(M_2,E)\right)^*, \\ \quad [\w] \ \longmapsto \overline{\langle \cdot \, , \ [\w]\rangle}_{P_2}\equiv \langle \cdot\, , \beta_{\pi}^*([\w]) \rangle_{P_2}, \\
\overline{\#}_{P_1}:H^{k+1}_{\textup{rel}}(M_1,E)\widetilde{\rightarrow} \left(H^{m-k-1}(\overline{(\Omega^*_{\max}(M_1,E),\D_1)})\right)^*, \\ \quad [\w] \ \longmapsto \overline{\langle \cdot \, , \ [\w]\rangle}_{P_1}\equiv \langle \beta_{\pi}'^*(\cdot) \, , [\w] \rangle_{P_1}.
\end{align*}
Due to commutativity of the previous two diagramms \eqref{comm1} and \eqref{comm2}, the commutativity of \eqref{third-type} is equivalent to commutativity of the following diagramm:
\begin{align}\label{third-type2}
\begin{xy}
  \xymatrix{
      H^k(\overline{(\Omega^*_{\max}(M_2,E),\D_2)})  \ar[r]^/.3em/{d^*} \ar[d]_{i^r(-1)^{\frac{k(k+1)}{2}}\overline{\#}_{P_2}}  & H^{k+1}_{\textup{rel}}(M_1,E) \ar[d]^{i^r(-1)^{\frac{(k+1)(k+2)}{2}}\overline{\#}_{P_1}}  \\
      H^{m-k}_{\textup{rel}}(M_2,E)^* \ar[r]^{d'_*} &  H^{m-k-1}(\overline{(\Omega^*_{\max}(M_1,E),\D_1)})^*. 
  }
\end{xy}
\end{align}
Using the explicit form of the connecting homomorphisms $d^*$ and $d'^*$ we finally compute for any $[(\w_1,\w_2)]\in H^k(\overline{(\Omega^*_{\max}(M_2,E),\D_2)})$ and $[(\eta_1,\eta_2)]\in H^{m-k-1}(\overline{(\Omega^*_{\max}(M_1,E),\D_1)})$:
\begin{align*}
\left(i^r(-1)^{\frac{(k+1)(k+2)}{2}}\overline{\#}_{P_1}\circ d^*\right)&[(\w_1,\w_2)]([(\eta_1,\eta_2)])-  \\ \left(i^r(-1)^{\frac{k(k+1)}{2}}d'_*\circ \overline{\#}_{P_2}\right)&[(\w_1,\w_2)]([(\eta_1,\eta_2)])= \\
=&i^r(-1)^{\frac{(k+1)(k+2)}{2}} \langle \beta_{\pi}'^*[(\eta_1,\eta_2)],d^*[(\w_1,\w_2)]\rangle_{P_1} - \\
&i^r(-1)^{\frac{k(k+1)}{2}} \langle d'^*[(\eta_1,\eta_2)],\beta_{\pi}^*[(\w_1,\w_2)]\rangle_{P_2}= \\
= &i^r(-1)^{\frac{(k+1)(k+2)}{2}} \int_{M_1}h^E(\eta_1\wedge \D_1\w_1)  - \\ - &i^r(-1)^{\frac{k(k+1)}{2}} \int_{M_2}h^ E(\D_2\eta_2\wedge \w_2)=:A.
\end{align*}
Now we apply the following formula for $j=1,2$:
$$dh^E(\eta_j\wedge \w_j)=h^E(\D_j\eta_j\wedge \w_j)+(-1)^{m-k-1}h^E(\eta_j\wedge \D_j\w_j).$$
Since $\D_1\eta_1=0$ and $\D_2\w_2=0$ we find 
\begin{align*}
A=i^r(-1)^{\frac{(k+1)(k+2)}{2}} \int_{M_1}(-1)^{m-k-1}dh^E(\eta_1\wedge \w_1)  - \\ i^r(-1)^{\frac{k(k+1)}{2}} \int_{M_2}dh^ E(\eta_2\wedge \w_2)=\\
i^r(-1)^{\frac{(k+1)(k+2)}{2}} (-1)^{m-k-1} \int_{\partial M_1}\iota_1^*h^E(\eta_1\wedge \w_1)  - \\ i^r(-1)^{\frac{k(k+1)}{2}} \int_{\partial M_2}\iota_2^*h^ E(\eta_2\wedge \w_2).
\end{align*}
Note $(-1)^{m-k-1}=(-1)^{-k}$ since $m$ is odd. Further $$\frac{(k+1)(k+2)}{2}-k=\frac{k(k+1)}{2}.$$ Hence we compute further
\begin{align}\label{calc1}
A=i^r(-1)^{\frac{k(k+1)}{2}+1}\left[\int_{\partial M_1}\iota_1^*h^E(\eta_1\wedge \w_1) + \int_{\partial M_2}\iota_2^*h^ E(\eta_2\wedge \w_2)\right].
\end{align}
Since $\iota_1^*\w_1=\iota_2^*\w_2$ and $\iota_1^*\eta_1=\iota_2^*\eta_2$ by construction, we find $$\iota_1^*h^E(\eta_1\wedge \w_1)= \iota_2^*h^ E(\eta_2\wedge \w_2).$$ However the orientations on $N=\partial M_1 = \partial M_2$ induced from $M_1$ and $M_2$ are opposite, thus the two integrals in \eqref{calc1} cancel. This shows commutativity of \eqref{third-type2} and completes the proof of the theorem.
\end{proof}

\section{Canonical Isomorphisms associated to Long Exact Sequences}\label{canonical} 
We first introduce some concepts and notations on finite-dimensional vector spaces. Let $V$ be an finite-dimensional complex vector space. Given a basis $\{v\}:=\{v_1,..,v_n\}, n=\dim V$, denote the induced element of the determinant line $\det V$ as follows
$$[v]:=v_1\wedge ..\wedge v_n\in \det V.$$
Given any two bases $\{v\}:=\{v_1,..,v_n\}$ and $\{w\}:=\{w_1,..,w_n\}$ of $V$, we have the corresponding coordinate change matrix $$v_i=\sum\limits_{j=1}^nl_{ij}w_j, \quad L:=(l_{ij}).$$ We put $$[v/w]:=\det L \in \C,$$ and obtain the following relation 
\begin{align}\label{coord-change}
[v]=[v/w][w].
\end{align}
In general the determinant is a complex number (we don't take the mode), but later it will be convenient to have a relation between bases such that the determinant of the coordinate change matrix is real-valued and positive. We will use the result of the following lemma.
\begin{lemma}\label{real-determinant}
Let $V$ be a complex finite-dimensional Hilbert space and $\{v\}$ any fixed basis, not necessarily orthogonal. Let $V=W\oplus W^{\perp}$ be an orthogonal decomposition into Hilbert subspaces. Then there exist orthonormal bases $\{w\}\equiv \{w_1,..w_{\dim W}\}, \{u\}\equiv \{u_1,..u_{\dim W^{\perp}}\}$ of $W,W^{\perp}$ respectively, such that the determinant of the coordinate change matrix between $\{w,u\}$ and $\{v\}$ is positive, i.e. $$[w,u/v]\in \R^+.$$
\end{lemma}
\begin{proof} Consider any orthonormal bases $\{w\}$ and $\{u\}$ of $W$ and $W^{\perp}$, respectively. This gives us two bases $\{v\}$ and $\{w,u\}$ of $V$. Denote the corresponding coordinate change matrix by $L$. We have $$[w,u/v]=\det L=e^{i\phi}|\det L|,$$ for some $\phi \in [0,2\pi)$. We replace $\{w\}$ and $\{u\}$ by  new bases 
\begin{align*}
\{w^v\}\equiv \{w^v_1,..,w^v_{\dim W}\}, \ w_i^v:=w_i\cdot \exp \left(\frac{-i\phi}{\dim V}\right), \\
\{u^v\}\equiv \{u^v_1,..,u^v_{\dim W^{\perp}}\}, \ u_i^v:=u_i\cdot \exp \left(\frac{-i\phi}{\dim V}\right).
\end{align*}
Note that $\{w^v\}$ and $\{u^v\}$ are still orthonormal bases of complex Hilbert spaces $W$ and $W^{\perp}$, respectively. By construction $[w^v,u^v/w,u]=\exp (-i\phi)$ and $$[w^v,u^v/v]=[w^v,u^v/w,u][w,u/v]=e^{-i\phi}\cdot e^{i\phi}|\det L|=|\det L|\in \R^+.$$
Thus $\{w^v,u^v\}$ indeed provides the desired example of an orthonormal basis of $V$, respecting the given orthogonal decomposition, with positive determinant of the coordinate change $[w^v,u^v/v]$ relative to any given basis $\{v\}$.
\end{proof} \ \\
\\[-7mm] The decomposition $V=W\oplus W^{\perp}$ in the lemma above is of course not essential for the statement itself. However we presented the result precisely in the form how it will be applied later. We will also need the following purely algebraic result:
\begin{prop}\label{dual-map}
Let $V$ and $W$ be two finite-dimensional Hilbert spaces with some orthonormal bases $\{v\}$ and $\{w\}$ respectively. Let $f:V\to W$ be an isomorphism of vector spaces. Then $\{f(v)\}$ is also a basis of $W$, not necessarily orthonormal. As Hilbert spaces $V$ and $W$ are canonically identified with their duals $V^*$ and $W^*$. Then $\{v^*\}, \{w^*\}$ are bases of $V^*,W^*$ respectively and $\{f^*(w^*)\}$ is another basis of $V^*$. Under this setup the following relation holds $$[f(v)/w]=[f^*(w^*)/v^*].$$
\end{prop}
\begin{proof}
Denote the scalar products on the Hilbert spaces $V$ and $W$ by $\langle\cdot \, , \cdot \rangle_V$ and $\langle \cdot \, , \cdot \rangle_W$, respectively. Let the scalar products be linear in the second component. They induce scalar products on $\det V$ and $\det W$, denoted by $\langle\cdot \, , \cdot \rangle_{\det V}$ and $\langle \cdot \, , \cdot \rangle_{\det W}$ respectively. Since the bases $\{v\},\{w\}$ are orthonormal, we obtain for the elements $[v], [w]$ of the determinant lines $\det V, \det W$ $$\langle[v] , [v] \rangle_{\det V}=\langle [w] , [w] \rangle_{\det W}=1.$$
The dual bases $\{v^*\},\{w^*\}$ induce elements on the determinant lines $\det V^*\cong (\det V)^*,$ and $\det W^*\cong (\det W)^*$ and under these identifications we have
\begin{align*}
[v^*]=[v]^*=\langle [v], \cdot \rangle_{\det V}, \\
[w^*]=[w]^*=\langle [w], \cdot \rangle_{\det W}.
\end{align*}
Now we compute
\begin{align*}
[f^*(w^*)]([v])=\langle [w], [f(v)]\rangle_{\det W}=[f(v)/w]\langle [w],[w]\rangle_{\det W}= \\ =[f(v)/w]\cdot 1=[f(v)/w]\cdot [v^*]([v]), \\
\Rightarrow [f^*(w^*)]=[f(v)/w][v^*].
\end{align*}
This implies the statement of the proposition.
\end{proof} \ \\
\\[-7mm] Next we consider the long exact sequences \eqref{LES-H}, introduced in Section \ref{cohom-algebra}. 
\begin{align*}&\mathcal{H}: \ 
...H^k_{\textup{rel}}(M_1,E)\xrightarrow{\A^*_k}H^k(M_1\#M_2,E)\xrightarrow{\beta^*_k}H^k_{\textup{abs}}(M_2,E)\xrightarrow{\delta^*_k}H^{k+1}_{\textup{rel}}(M_1,E)... \\ &\mathcal{H'}\!: ..H^k_{\textup{rel}}(M_2,E)\xrightarrow{\A'^*_k}H^k(M_1\#M_2,E)\xrightarrow{\beta'^*_k}H^k_{\textup{abs}}(M_1,E)\xrightarrow{\delta'^*_k}H^{k+1}_{\textup{rel}}(M_2,E)...
\end{align*}
The long exact sequences induce isomorphisms on determinant lines (cf. [Nic]) in a canonical way
\begin{align}
\Phi: \det  H^*_{\textup{rel}}(M_1,E) \otimes \det H^*_{\textup{abs}}(M_2,E) \otimes \left[\det H^*(M_1\#M_2,E)\right]^{-1}\rightarrow \C, \\
\Phi': \det  H^*_{\textup{rel}}(M_2,E) \otimes \det H^*_{\textup{abs}}(M_1,E) \otimes \left[\det H^*(M_1\#M_2,E)\right]^{-1}\rightarrow \C.
\end{align}
More precisely , the action of the isomorphisms $\Phi,\Phi'$ is explicitly given as follows. Fix any bases $\{\widetilde{a}_k\}$, $\{\widetilde{b}_k\}$ and $\{\widetilde{c}_k\}$ of $\textup{Im} \delta^*_{k-1}$, $\textup{Im} \A^*_{k}$ and $\textup{Im} \beta^*_{k}$ respectively. Here the lower index $k$ indexes the entire basis and is not a counting of the elements in the set. Choose now any linearly independent elements $\{a_k\}$,  $\{b_k\}$ and $\{c_k\}$ such that $\{\widetilde{a}_k\}=\delta_{k-1}^*(c_{k-1})$, $\{\widetilde{b}_k\}=\A^*_k (a_k)$ and $\{\widetilde{c}_k\}=\beta_{k}^*(b_k)$.
\\[3mm] We make the same choices on the long exact sequence $\mathcal{H'}$. The notation is the same up to an additional apostroph. Since the sequences $\mathcal{H}, \mathcal{H'}$ are exact, the choices above provide us with bases of the cohomology spaces.
\\[3mm] Under the Knudson-Mumford sign convention [KM] we define the action of the isomorphisms $\Phi$ and $\Phi'$ as follows:
\begin{align}
\Phi \left\{\left(\bigotimes\limits_{k=0}^m[a_k,\widetilde{a}_k]^{(-1)^k}\right)\otimes \left(\bigotimes\limits_{k=0}^m[c_k,\widetilde{c}_k]^{(-1)^k}\right) \otimes \left(\bigotimes\limits_{k=0}^m[b_k,\widetilde{b}_k]^{(-1)^{k+1}}\right) \right\}\mapsto (-1)^{\nu}, \label{action1} \\
\Phi' \left\{\left(\bigotimes\limits_{k=0}^m[a'_k,\widetilde{a}'_k]^{(-1)^k}\right)\otimes \left(\bigotimes\limits_{k=0}^m[c'_k,\widetilde{c}'_k]^{(-1)^k}\right) \otimes \left(\bigotimes\limits_{k=0}^m[b'_k,\widetilde{b}'_k]^{(-1)^{k+1}}\right) \right\}\mapsto (-1)^{\nu'}.\label{action2}
\end{align}
The definition turns out to be independent of choices. The numbers $\nu,\nu'$ count the pairwise reorderings in the definition of the isomorphisms. They are given explicitly by the following formula:
\begin{align}\nonumber
\nu = \frac{1}{2}\sum\limits_{k=0}^m
\left(\dim \textup{Im}\A^*_k\cdot (\dim \textup{Im}\A^*_k +(-1)^k)\right)+ \\ \nonumber
\frac{1}{2}\sum\limits_{k=0}^m
\left(\dim \textup{Im}\beta^*_k\cdot (\dim \textup{Im}\beta^*_k +(-1)^k)\right)+ \\ \nonumber
\frac{1}{2}\sum\limits_{k=0}^m
\left(\dim \textup{Im}\delta^*_k\cdot (\dim \textup{Im}\delta^*_k +(-1)^k)\right)+ \\ \nonumber
\sum\limits_{k=0}^m\left(\dim H^k_{\textup{rel}}(M_1,E)\cdot \sum\limits_{i=0}^{k-1}\dim H^i(M_1\#M_2,E) \right) + \\ \nonumber
\sum\limits_{k=0}^m\left(\dim H^k_{\textup{rel}}(M_1,E)\cdot \sum\limits_{i=0}^{k-1}\dim H^i_{\textup{abs}}(M_2,E) \right) + \\ \label{v}
\sum\limits_{k=0}^m\left(\dim H^i_{\textup{abs}}(M_2,E)\cdot \sum\limits_{i=0}^{k-1}\dim H^i_{\textup{abs}}(M_2,E) \right).
\end{align}
The first three lines in the formula are standard terms for "cancellations" of images and cokernels of the homomorphisms in an acyclic sequence of vector spaces. The last three lines are due to reordering of the cohomology groups into determinant lines. The number $\nu'$ is given by an analogous formula as $\nu$. As a consequence of Theorem \ref{comm-diagramm2} which relates both sequences $\mathcal{H}$ and $\mathcal{H'}$ we have $$\nu =\nu'.$$
Let the cohomology spaces in the long exact sequences $\mathcal{H}$ and $\mathcal{H'}$ be endowed with Hilbert structures naturally induced by the $L^2-$scalar products on harmonic elements. We have an orthogonal decomposition of each cohomology space in the long exact sequences:
\begin{align}
&H^k_{\textup{rel}}(M_1,E)=\textup{Im}\delta^*_{k-1}\oplus (\textup{Im}\delta^*_{k-1})^{\perp}, \nonumber \\ 
&H^k(M_1\#M_2,E)= \textup{Im}\A^*_{k}\oplus (\textup{Im}\A^*_{k})^{\perp}, \nonumber \\
&H^k_{\textup{abs}}(M_2,E)= \textup{Im}\beta^*_{k}\oplus (\textup{Im}\beta^*_{k})^{\perp}, \nonumber \\
& \qquad \qquad \qquad H^k_{\textup{rel}}(M_2,E)= \textup{Im}\delta'^*_{k-1}\oplus (\textup{Im}\delta'^*_{k-1})^{\perp}, \nonumber \\
& \qquad \qquad \qquad H^k(M_1\#M_2,E)= \textup{Im}\A'^*_{k}\oplus (\textup{Im}\A'^*_{k})^{\perp}, \nonumber \\
& \qquad \qquad \qquad H^k_{\textup{abs}}(M_1,E)= \textup{Im}\beta'^*_{k}\oplus (\textup{Im}\beta'^*_{k})^{\perp}. \nonumber \\
\ \label{orth-dec}
\end{align}
We can assume the bases $\{a_k,\widetilde{a}_k\}, \{b_k,\widetilde{b}_k\}, \{c_k,\widetilde{c}_k\}$ on $\mathcal{H}$ as well as the corresponding bases on $\mathcal{H'}$ to respect the orthogonal decomposition above, i.e. with respect to the orthogonal decompositions in \eqref{orth-dec} we have
\begin{align*}
&H^k_{\textup{rel}}(M_1,E)=\langle \{\widetilde{a}_k\}\rangle \oplus \langle \{ a_k\}\rangle, \\
&H^k(M_1\#M_2,E)=\langle \{\widetilde{b}_k\}\rangle \oplus \langle \{ b_k\}\rangle,\\
&H^k_{\textup{abs}}(M_2,E)= \langle \{\widetilde{c}_k\}\rangle \oplus \langle \{c_k\}\rangle, \\
& \qquad \qquad \qquad H^k_{\textup{rel}}(M_2,E)=\langle \{\widetilde{a}'_k\}\rangle \oplus \langle \{ a'_k\}\rangle, \\
& \qquad \qquad \qquad H^k(M_1\#M_2,E)=\langle \{\widetilde{b}'_k\}\rangle \oplus \langle \{ b'_k\}\rangle,\\
& \qquad \qquad \qquad H^k_{\textup{abs}}(M_1,E)= \langle \{\widetilde{c}'_k\}\rangle \oplus \langle \{c'_k\}\rangle.
\end{align*}
By Lemma \ref{real-determinant} we can choose for any $k=0,..,\dim M$ orthonormal bases of $H^k_{\textup{rel}}(M_1,E), H^k(M_1\#M_2,E), H^k_{\textup{abs}}(M_2,E)$ with respect to orthogonal decomposition \eqref{orth-dec}
\begin{align}
&H^k_{\textup{rel}}(M_1,E)=\langle \{\widetilde{v}_k\}\rangle \oplus \langle \{v_k\}\rangle,  \nonumber \\
&H^k(M_1\#M_2,E)=\langle \{\widetilde{w}_k\}\rangle \oplus \langle \{ w_k\}\rangle, \nonumber \\
&H^k_{\textup{abs}}(M_2,E)= \langle \{\widetilde{u}_k\}\rangle \oplus \langle \{u_k\}\rangle, \nonumber \\ \ \label{basis-choice}
\end{align} 
such that 
\begin{align}\label{pos-det}
[v_k, \widetilde{v}_k/a_k, \widetilde{a}_k], \ [u_k, \widetilde{u}_k /c_k, \widetilde{c}_k], \ [w_k, \widetilde{w}_k/b_k, \widetilde{b}_k] \ \in \R^+.
\end{align}
These bases induce bases of the cohomology spaces of the sequence $\mathcal{H'}$ by the action of the Poincare duality map $\G$. Since the map is an isometry, the induced bases are still orthonormal. Furthermore commutativity of the diagramm \eqref{comm-diagramm}, established in Theorem \ref{comm-diagramm2} implies that the induced bases still respect the orthogonal decomposition \eqref{orth-dec} of the cohomology spaces.   
\begin{align*}
& \qquad \qquad \qquad H^{m-k}_{\textup{rel}}(M_2,E)=\langle \{\G u_k\}\rangle \oplus \langle \{\G \widetilde{u}_k\}\rangle, \\
& \qquad \qquad \qquad H^{m-k}(M_1\#M_2,E)=\langle \{\G w_k\}\rangle \oplus \langle \{ \G \widetilde{w}_k\}\rangle,\\
& \qquad \qquad \qquad H^{m-k}_{\textup{abs}}(M_1,E)= \langle \{\G v_k\}\rangle \oplus \langle \{\G \widetilde{v}_k\}\rangle.
\end{align*}
We obtain for the action of the canonical isomorphisms on the elements induced by these orthonormal bases the following central result, which relates the action of the isomorphisms to the combinatorial torsion of the long exact sequences. 
\begin{thm}\label{phi-tau}
\begin{align*}
&\Phi \left\{\left(\bigotimes\limits_{k=0}^m[v_k,\widetilde{v}_k]^{(-1)^k}\right)\otimes \left(\bigotimes\limits_{k=0}^m[u_k,\widetilde{u}_k]^{(-1)^k}\right) \otimes \left(\bigotimes\limits_{k=0}^m[w_k,\widetilde{w}_k]^{(-1)^{k+1}}\right) \right\}= \\
&\Phi' \left\{\left(\bigotimes\limits_{k=0}^m[\G \widetilde{v}_k, \G v_k]^{(-1)^{m-k}}\right)\otimes \left(\bigotimes\limits_{k=0}^m[\G \widetilde{u}_k, \G u_k]^{(-1)^{m-k}}\right) \otimes \right.\\ &\qquad \qquad \qquad \left. \otimes \left(\bigotimes\limits_{k=0}^m[\G \widetilde{w}_k, \G w_k]^{(-1)^{m-k+1}}\right) \right\}=  (-1)^{\nu}\cdot \tau(\mathcal{H})=(-1)^{\nu}\tau (\mathcal{H'}).
\end{align*}
\end{thm}
\begin{remark}
The statement of the theorem corresponds to the fact that the combinatorial torsions $\tau(\mathcal{H}), \tau(\mathcal{H'})$ are defined as modes of the complex numbers obtained by the action of the isomorphisms $\Phi,\Phi'$ on the volume elements, induced by the Hilbert structures.
\\[3mm] However the value of the theorem for our purposes is firstly the equality $\tau(\mathcal{H})=\tau (\mathcal{H'})$ and most importantly the fact that it provides explicit volume elements on the determinant lines, which are mapped to the real-valued positive combinatorial torsions without additional undetermined complex factors of the form $e^{i\phi}$.
\end{remark} \ \\
\\[-7mm] \emph{Proof of Theorem \ref{phi-tau}.}
Consider first the action of the canonical isomorphism $\Phi$. By the action \eqref{action1} we obtain
\begin{align} \nonumber
\Phi \left\{\left(\bigotimes\limits_{k=0}^m[v_k,\widetilde{v}_k]^{(-1)^k}\right)\otimes \left(\bigotimes\limits_{k=0}^m[u_k,\widetilde{u}_k]^{(-1)^k}\right) \otimes \left(\bigotimes\limits_{k=0}^m[w_k,\widetilde{w}_k]^{(-1)^{k+1}}\right) \right\}= \\ \nonumber
(-1)^{\nu}\prod_{k=0}^m [v_k,\widetilde{v}_k/ a_k,\widetilde{a}_k]^{(-1)^k}\cdot[u_k,\widetilde{u}_k/ c_k,\widetilde{c}_k]^{(-1)^k}\cdot[w_k,\widetilde{w}_k/ b_k,\widetilde{b}_k]^{(-1)^{k+1}}=\\ \label{tau} =(-1)^{\nu}\tau (\mathcal{H}),
\end{align}
where the second equality follows from the definition of combinatorial torsion and the particular choice of bases such that \eqref{pos-det} holds. On the other hand we can rewrite the action of $\Phi$ as follows:
\begin{align}
\Phi \left\{\left(\bigotimes\limits_{k=0}^m[v_k,\widetilde{v}_k]^{(-1)^k}\right)\otimes \left(\bigotimes\limits_{k=0}^m[u_k,\widetilde{u}_k]^{(-1)^k}\right) \otimes \left(\bigotimes\limits_{k=0}^m[w_k,\widetilde{w}_k]^{(-1)^{k+1}}\right) \right\}= \nonumber \\
(-1)^{\nu}\prod\limits_{k=0}^m[v_k,\widetilde{v}_k/ a_k,\widetilde{a}_k]^{(-1)^k}\cdot[u_k,\widetilde{u}_k/ c_k,\widetilde{c}_k]^{(-1)^k}\cdot[w_k,\widetilde{w}_k/ b_k,\widetilde{b}_k]^{(-1)^{k+1}}= \nonumber \\
(-1)^{\nu}\prod\limits_{k=0}^m [v_k/ a_k]^{(-1)^k}[\widetilde{v}_k/\widetilde{a}_k]^{(-1)^k}\cdot[u_k/ c_k]^{(-1)^k}\cdot[\widetilde{u}_k/ \widetilde{c}_k]^{(-1)^k} \nonumber \\ \cdot [w_k/ b_k]^{(-1)^{k+1}}\cdot[\widetilde{w}_k / \widetilde{b}_k]^{(-1)^{k+1}}. \label{phi-action}
\end{align}
Observe now the following useful relations:
\begin{align*}
[\A^*_k(v_k)]=[\A^*_k(v_k)&/\A^*_k(a_k)][\A^*_k(a_k)]=[v_k/a_k][\widetilde{b}_k]=\frac{[v_k /a_k]}{[\widetilde{w}_k/\widetilde{b}_k]}\cdot [\widetilde{w}_k] \\
\textup{and hence } \  \, &[\A^*_k(v_k)/\widetilde{w}_k]=\frac{[v_k /a_k]}{[\widetilde{w}_k/\widetilde{b}_k]}, \\
&[\beta^*_k(w_k)/\widetilde{u}_k]=\frac{[w_k/b_k]}{[\widetilde{u}_k / \widetilde{c}_k]}, \\
&[\delta^*_k(u_k)/ \widetilde{v}_{k+1}]=\frac{[u_k/c_k]}{[\widetilde{v}_{k+1}/\widetilde{a}_{k+1}]},
\end{align*}
where the last two identities are derived in the similar manner as the first one. With these relations we can rewrite the action \eqref{phi-action} of $\Phi$ as follows:
\begin{align}
\Phi \left\{\left(\bigotimes\limits_{k=0}^m[v_k,\widetilde{v}_k]^{(-1)^k}\right)\otimes \left(\bigotimes\limits_{k=0}^m[u_k,\widetilde{u}_k]^{(-1)^k}\right) \otimes \left(\bigotimes\limits_{k=0}^m[w_k,\widetilde{w}_k]^{(-1)^{k+1}}\right) \right\}= \nonumber \\
(-1)^{\nu}\prod\limits_{k=0}^m [\A^*_k(v_k)/\widetilde{w}_k]^{(-1)^k}\cdot [\beta^*_k(w_k)/\widetilde{u}_k]^{(-1)^{k+1}}\cdot [\delta^*_k(u_k)/ \widetilde{v}_{k+1}]^{(-1)^k}. \label{phi-action2}
\end{align}
Analogous argumentation for the canonical isomorphism $\Phi'$ shows 
\begin{align}
\Phi' \left\{\left(\bigotimes\limits_{k=0}^m[\G \widetilde{v}_{k}, \G v_{k}]^{(-1)^{m-k}}\right)\otimes \right. \qquad \qquad \qquad \qquad \\ \otimes \left. \left(\bigotimes\limits_{k=0}^m[\G \widetilde{u}_{k}, \G u_{k}]^{(-1)^{m-k}}\right) \otimes \left(\bigotimes\limits_{k=0}^m[\G \widetilde{w}_{k}, \G w_{k}]^{(-1)^{k}}\right) \right\}\nonumber \\
=(-1)^{\nu}\!\!\prod\limits_{k=0}^m [\A'^*_{m-k}(\G \widetilde{u}_{k})/\G w_{k}]^{(-1)^{m-k}}\cdot [\beta'^*_{m-k}(\G \widetilde{w}_{k})/\G v_{k}]^{(-1)^{k}} \nonumber \\ \cdot [\delta'^*_{m-k}(\G \widetilde{v}_{k})/ \G u_{k-1}]^{(-1)^{m-k}}. \nonumber \\ \ \label{phi-action3}
\end{align}
Now using the fact that the diagramm \eqref{comm-diagramm} is commutative with vertical maps being linear, we obtain
\begin{align*}
[\delta^*_k(u_k)/ \widetilde{v}_{k+1}]=[(\#_{L^2_1}\circ \G )\delta^*_k(u_k)/ (\#_{L^2_1}\circ \G )\widetilde{v}_{k+1}]= [\delta_*'^{m-k}(\G u_{k})^*/(\G \widetilde{v}_{k+1})^*], \\
[\A^*_k(v_k)/\widetilde{w}_k]=[\beta'^{m-k}_*(\G v_k)^*/(\G\widetilde{w}_k)^*], \\
[\beta^*_k(w_k)/\widetilde{u}_k]= [\A'^{m-k}_*(\G w_k)^*/(\G\widetilde{u}_k)^*],
\end{align*}
where the last two identities are derived in a similar manner as the first one. Now with the following purely algebraic result of Proposition \ref{dual-map} we obtain
\begin{align*}
[\delta^*_k(u_k)/ \widetilde{v}_{k+1}]= [\delta_*'^{m-k}(\G u_{k})^*/(\G \widetilde{v}_{k+1})^*]= [\delta'^*_{m-k}(\G \widetilde{v}_{k+1})/(\G u_{k})], \\
[\A^*_k(v_k)/\widetilde{w}_k]=[\beta'^{m-k}_*(\G v_k)^*/(\G\widetilde{w}_k)^*]=[\beta'^*_{m-k}(\G\widetilde{w}_k)/(\G v_k)], \\
[\beta^*_k(w_k)/\widetilde{u}_k]= [\A'^{m-k}_*(\G w_k)^*/(\G\widetilde{u}_k)^*]= [\A'^*_{m-k}(\G\widetilde{u}_k)/(\G w_k)].
\end{align*}
These identities allow us to compare the actions \eqref{phi-action2} and \eqref{phi-action3} and derive the equality:
\begin{align}\nonumber
&\Phi \left\{\left(\bigotimes\limits_{k=0}^m[v_k,\widetilde{v}_k]^{(-1)^k}\right)\otimes \left(\bigotimes\limits_{k=0}^m[u_k,\widetilde{u}_k]^{(-1)^k}\right) \otimes \left(\bigotimes\limits_{k=0}^m[w_k,\widetilde{w}_k]^{(-1)^{k+1}}\right) \right\}= \\ \nonumber
=&\Phi' \left\{\left(\bigotimes\limits_{k=0}^m[\G \widetilde{v}_k, \G v_k]^{(-1)^{m-k}}\right)\otimes \left(\bigotimes\limits_{k=0}^m[\G \widetilde{u}_k, \G u_k]^{(-1)^{m-k}}\right) \otimes \right.\\ \label{tau2} &\qquad \qquad \qquad \left. \otimes \left(\bigotimes\limits_{k=0}^m[\G \widetilde{w}_k, \G w_k]^{(-1)^{m-k+1}}\right) \right\}=  (-1)^{\nu}\cdot \tau(\mathcal{H}).
\end{align}
On the other hand, since $\G$ is an isometry, we have in \eqref{tau2} the $\Phi'$-action on a volume element, induced by the Hilbert structures on $\mathcal{H'}$. The combinatorial torsion $\tau (\mathcal{H'})$ is defined as the mode of the complex-valued $\Phi'$-image of the volume element. Hence
\begin{align}
\Phi' \left\{\left(\bigotimes\limits_{k=0}^m[\G \widetilde{v}_k, \G v_k]^{(-1)^{m-k}}\!\right)\otimes \right. \qquad \qquad \qquad \\ \otimes \left. \left(\bigotimes\limits_{k=0}^m[\G \widetilde{u}_k, \G u_k]^{(-1)^{m-k}}\!\right) \otimes \left(\bigotimes\limits_{k=0}^m[\G \widetilde{w}_k, \G w_k]^{(-1)^{k}}\!\right) \right\} \nonumber \\ =\label{tau3} (-1)^{\nu}e^{i\psi}\cdot \tau(\mathcal{H'}).
\end{align}
The phase $e^{i\psi}$ can be viewed as the total rotation angle needed to rotate the orthonormal bases $\{\G \widetilde{v}_k, \G v_k \}, \{\G \widetilde{u}_k, \G u_k\}, \{\G \widetilde{w}_k, \G w_k\}$ to orthonormal bases with positive determinants of coordinate change matrices with respect to bases fixed in \eqref{action2} (cf. Lemma \ref{real-determinant}).  
\\[3mm] Since the combinatorial torsions are positive real numbers, comparison of \eqref{tau3} with \eqref{tau2} leads to $$\tau(\mathcal{H})=\tau(\mathcal{H'}).$$
This completes the statement of the theorem. $\square$
\\[4mm] The canonical isomorphisms $\Phi,\Phi'$ induce isomorphisms 
\begin{align}
\Psi: \det  H^*_{\textup{rel}}(M_1,E) \otimes \det H^*_{\textup{abs}}(M_2,E) \rightarrow \det H^*(M_1\#M_2,E), \label{psi1}\\
\Psi': \det  H^*_{\textup{rel}}(M_2,E) \otimes \det H^*_{\textup{abs}}(M_1,E) \rightarrow \det H^*(M_1\#M_2,E)\label{psi2}
\end{align}
by the following formula. Consider any $x \in \det  H^*_{\textup{rel}}(M_1,E)$, $y \in \det H^*_{\textup{abs}}(M_2,E)$ and $z \in \det H^*(M_1\#M_2,E)$. Then we set
$$\Psi(x \otimes y):=\Phi (x\otimes y \otimes z^{-1}) z.$$
The definition of $\Psi'$ is analogous. Then with the result and notation of Theorem \ref{phi-tau} we obtain:
\begin{cor}\label{phi-tau-2}
\begin{align*}
\Psi \left\{\left(\bigotimes\limits_{k=0}^m[v_k,\widetilde{v}_k]^{(-1)^k}\right)\otimes \left(\bigotimes\limits_{k=0}^m[u_k,\widetilde{u}_k]^{(-1)^k}\right) \right\} =\\=(-1)^{\nu}\tau(\mathcal{H})\left(\bigotimes\limits_{k=0}^m[w_k,\widetilde{w}_k]^{(-1)^{k}}\right), \\
\Psi' \left\{\left(\bigotimes\limits_{k=0}^m[\G \widetilde{v}_k, \G v_k]^{(-1)^{m-k}}\right)\otimes \left(\bigotimes\limits_{k=0}^m[\G \widetilde{u}_k, \G u_k]^{(-1)^{m-k}}\right) \right\}=\\  =(-1)^{\nu} \tau(\mathcal{H})\left(\bigotimes\limits_{k=0}^m[\G \widetilde{w}_k, \G w_k]^{(-1)^{m-k}}\right).
\end{align*}
\end{cor}

\section{Splitting formula for Refined Torsion of complexes}\label{splitting-refined-torsion} 
We continue in the setup of Section \ref{gluing-statement}. Consider the refined analytic torsions of the manifolds $M_j,j=1,2$ and the split manifold $M=M_1\cup_N M_2$. According to Proposition \ref{rho-formula} we can write for the refined analytic torsions 
\begin{align*}
\rho_{\textup{an}}(\D)=\frac{1}{T^{RS}(\DD)}\cdot \exp \left[-i\pi \eta(\B_{\textup{even}})+i\pi \textup{rk}(E) \eta (\B_{\textup{trivial}})\right] \hspace{30mm}\\
\times \exp\left[-i\pi \frac{m-1}{2}\dim \ker \B_{\textup{even}}+i\pi \textup{rk}(E) \frac{m}{2}\dim \ker \B_{\textup{trivial}}\right]\rho_{\G}(M,E),\\
\rho_{\textup{an}}(\D_j)=\frac{1}{T^{RS}(\DD_j)}\cdot \exp \left[-i\pi \eta(\B^j_{\textup{even}})+i\pi \textup{rk}(E) \eta (\B^j_{\textup{trivial}})\right] \hspace{30mm} \\ \times \exp\left[-i\pi \frac{m-1}{2}\dim \ker \B^j_{\textup{even}}+i\pi \textup{rk}(E) \frac{m}{2}\dim \ker \B^j_{\textup{trivial}}\right]\rho_{\G}(M_j,E),
\end{align*}
where $j=1,2$ and $T^{RS}(\DD), T^{RS}(\DD_j)$ denote the scalar analytic torsions associated to the complexes $(\domm, \DD), (\domm_j, \DD_j)$ respectively. Furthermore $\rho_{\G}(M,E), \rho_{\G}(M_j,E)$ denote the respective refined torsion elements in the sense of [BV3, (3.7)] for $\lambda=0$. The refined torsion elements are elements of the determinant lines:
\begin{align*}
&\rho_{\G}(M_1,E)\in \det (H^*_{\textup{rel}}(M_1,E)\oplus H^*_{\textup{abs}}(M_1,E)), \\
&\rho_{\G}(M_2,E)\in \det (H^*_{\textup{rel}}(M_2,E)\oplus H^*_{\textup{abs}}(M_2,E)), \\
&\rho_{\G}(M,E)\in \det (H^*(M,E)\oplus H^*(M,E)).
\end{align*}
These elements are in the sense of [BK2, Section 4] the refined torsions $\rho_{[0,\lambda]},\lambda =0$ (see also [BV3, (3.7)]) of the corresponding complexes: 
\begin{align*}
H^*_{\textup{rel}}(M_1,E)&\oplus H^*_{\textup{abs}}(M_1,E), \\
H^*_{\textup{rel}}(M_2,E)&\oplus H^*_{\textup{abs}}(M_2,E), \\
H^*(M,E)&\oplus H^*(M,E).
\end{align*}
Note that up to the identification of Corollary \ref{split-cohomology} the refined torsion $\rho_{\G}(M,E)$ corresponds to the refined torsion of the complex $H^*(M_1\#M_2,E)\oplus H^*(M_1\#M_2,E)$
\begin{align}\label{split-refined2}
\rho_{\G}(M_1\#M_2,E)\in \det (H^*(M_1\#M_2,E)\oplus H^*(M_1\#M_2,E)).
\end{align}
With the preceeding three sections we can now relate the refined torsions $\rho_{\G}(M_1,E)$, $\rho_{\G}(M_2,E)$ and $\rho_{\G}(M_1\#M_2,E)$ together.
For this we first rewrite the refined torsions in convenient terms. We restrict the neccessary arguments to $\rho_{\G}(M_1,E)$, since the discussion of the other elements is completely analogous. 
\\[3mm] Let for $k=0,..,\dim M$ the sets $\{e_k\}$ and $\{\theta_k\}$ be the bases for $H^k_{\textup{rel}}(M_1,E)$ and $H^k_{\textup{abs}}(M_1,E)$ respectively. Then the refined torsion element $\rho_{\G}(M_1,E)$ is given by:
\begin{align}
\rho_{\G}(M_1,E)=(-1)^{R_1}([e_0]\wedge [\theta_0])\otimes ([e_1]\wedge [\theta_1])^{(-1)}\otimes ... \nonumber \\ ...\otimes ([e_{r-1}]\wedge [\theta_{r-1}])^{(-1)^{r-1}}\otimes ([\G \theta_{r-1}]\wedge [\G e_{r-1}])^{(-1)^{r}}\otimes ...\nonumber \\ \label{refined-element-explicit}  ...\otimes ([\G \theta_{1}]\wedge [\G e_{1}])\otimes ([\G \theta_{0}]\wedge [\G e_{0}])^{(-1)},
\end{align}
where $r=(\dim M +1)/2$. The sign $R_1$ is given according to [BK2, (4.2)] by 
\begin{align*}
R_1=\frac{1}{2}\sum\limits_{k=0}^{r-1}&(\dim H^k_{\textup{rel}}(M_1,E) + \dim H^k_{\textup{abs}}(M_1,E)) \cdot \\ \cdot &\left(\dim H^k_{\textup{rel}}(M_1,E) + \dim H^k_{\textup{abs}}(M_1,E) +(-1)^{r-k}\right).
\end{align*}
The formula for $\rho_{\G}(M_1,E)$ is independent of the particular choice of bases $\{e_k\}$ and $\{\theta_k\}$. Hence, since $\{\G e_k\}$ is also a basis of $H^{m-k}_{\textup{abs}}(M_1,E)$ for any $k$, we can write equivalently, replacing in the formula \eqref{refined-element-explicit} the basis $\{\theta_k\}$ by $\{\G e_{m-k}\}$:
\begin{align*}
\rho_{\G}(M_1,E)=(-1)^{R_1}([e_0]\wedge [\G e_m])\otimes ([e_1]\wedge [\G e_{m-1}])^{(-1)}\otimes ... \qquad \qquad \qquad \\ ...\otimes (e_{m-1}]\wedge [\G e_{1}])\otimes ([e_m]\wedge [\G e_{0}])^{(-1)}.
\end{align*}
With the "fusion isomorphism" for graded vector spaces (cf. [BK2, (2.18)])
\begin{align*}
\mu_{(M_1,E)}:\det H^*_{\textup{rel}}(M_1,E) \otimes \det H^*_{\textup{abs}}&(M_1,E) \\ &\xrightarrow{\sim} \det (H^*_{\textup{rel}}(M_1,E) \oplus \det H^*_{\textup{abs}}(M_1,E))
\end{align*}
we obtain
\begin{align*}
\mu_{(M_1,E)}^{(-1)}\left(\rho_{\G}(M_1,E)\right)=\left(\bigotimes\limits_{k=0}^m [e_k]^{(-1)^k} \right)\otimes \left(\bigotimes\limits_{k=0}^m [\G e_k]^{(-1)^{m-k}}\right)\cdot (-1)^{\mathcal{M}(M_1,E)+R_1},
\end{align*}
where with [BK2, (2.19)]
$$\mathcal{M}(M_1,E)\,=\!\sum_{0\leq k<i\leq m}\dim H^i_{\textup{rel}}(M_1,E)\cdot \dim  H^k_{\textup{abs}}(M_1,E).$$
Analogous result holds for the refined torsions $\rho_{\G}(M_2,E)$ and $\rho_{\G}(M_1\#M_2,E)$, where the analogous quantities $R,R_2$ and $\mathcal{M}(M,E)$ and $\mathcal{M}(M_2,E)$ are introduced respectively. Using now the fact that the refined torsion elements are independent of choices, we find with bases, fixed in \eqref{basis-choice}:
\begin{align*}
&\mu_{(M_1,E)}^{(-1)}\left(\rho_{\G}(M_1,E)\right)= \\  & \hspace{22mm}=(-1)^{\mathcal{M}(M_1,E)+R_1} \left( \bigotimes\limits_{k=0}^m[v_k,\widetilde{v}_k]^{(-1)^{k}} \right) \otimes \left( \bigotimes\limits_{k=0}^m[\G \widetilde{v}_k,\G v_k]^{(-1)^{m-k}} \right), \\
&\mu_{(M_2,E)}^{(-1)}\left(\rho_{\G}(M_2,E)\right)= \\ & \hspace{20mm}=(-1)^{\mathcal{M}(M_2,E)+R_2} \left( \bigotimes\limits_{k=0}^m[u_k,\widetilde{u}_k]^{(-1)^{k}} \right) \otimes \left( \bigotimes\limits_{k=0}^m[\G \widetilde{u}_k,\G u_k]^{(-1)^{m-k}} \right), \\
&\mu_{(M_1\#M_2,E)}^{(-1)}\left(\rho_{\G}(M_1\#M_2,E)\right)= \\ & \hspace{14mm} =(-1)^{\mathcal{M}(M_1\#M_2,E)+R} \left( \bigotimes\limits_{k=0}^m[w_k,\widetilde{w}_k]^{(-1)^{k}} \right) \otimes \left( \bigotimes\limits_{k=0}^m[\G \widetilde{w}_k,\G w_k]^{(-1)^{m-k}} \right).
\end{align*}
Now combine the canonical isomorphisms $\Psi, \Psi'$, introduced in \eqref{psi1} and \eqref{psi2}, together with the fusion isomorphisms into one single canonical isomorphism:
\begin{align}\label{omega}
\Omega :=&\mu_{(M_1\#M_2,E)}\circ (\Psi\otimes \Psi')\circ (\mu_{(M_1,E)}^{-1}\otimes \mu_{(M_2,E)}^{-1}):\\ \nonumber
&\det (H^*_{\textup{rel}}(M_1,E)\oplus H^*_{\textup{abs}}(M_1,E)) \otimes \\ \nonumber &\det (H^*_{\textup{rel}}(M_2,E)\oplus H^*_{\textup{abs}}(M_2,E)) \rightarrow \\ \nonumber
& \hspace{42mm} \rightarrow \det (H^*(M_1\#M_2,E)\oplus H^*(M_1\#M_2,E)),
\end{align}
where we employed implicitly flip-isomorphisms in order to reorder the determinant lines appropriately. Due to the Knudson-Momford sign convention this leads to an additional sign. We obtain by Corollary \ref{phi-tau-2} for the action of this canonical isomorphism
\begin{align*}
\Omega (\rho_{\G}(M_1,E)\otimes \rho_{\G}(M_2,E))=(-1)^{\mathcal{M}(M_1,E)+\mathcal{M}(M_2,E)+R_1+R_2+1} \times \\
\mu_{(M_1\#M_2,E)}\left( \Psi \left[ \left( \bigotimes\limits_{k=0}^m[v_k,\widetilde{v}_k]^{(-1)^{k}} \right) \otimes \left( \bigotimes\limits_{k=0}^m[u_k, \widetilde{u}_k]^{(-1)^{k}} \right)\right]\otimes \right. \\ \left. \Psi'\left[ \left( \bigotimes\limits_{k=0}^m[\G \widetilde{v}_k,\G v_k]^{(-1)^{k+1}} \right) \otimes \left( \bigotimes\limits_{k=0}^m[\G \widetilde{u}_k,\G u_k]^{(-1)^{k+1}} \right)\right]\right)= \\
= (-1)^{\mathcal{M}(M_1,E)+\mathcal{M}(M_2,E)+R_1+R_2+1} \tau (\mathcal{H})^2\times \\ \mu_{(M_1\#M_2,E)}\left(\bigotimes\limits_{k=0}^m[w_k,\widetilde{w}_k]^{(-1)^{k}} \right) \otimes \left( \bigotimes\limits_{k=0}^m[\G \widetilde{w}_k,\G w_k]^{(-1)^{k+1}} \right) =\\
= (-1)^{\textup{sign}} \tau (\mathcal{H})^2\rho_{\G}(M_1\#M_2,E),
\end{align*}
where we have set
\begin{align}\label{sign}
\textup{sign}:= \mathcal{M}(M_1,E)+\mathcal{M}(M_2,E)-\mathcal{M}(M_1\#M_2,E) + R_1+R_2-R+1.
\end{align}
Summarizing, we have derived a relation between the refined torsion elements of the splitting problem under the canonical isomorphism $\Omega$:
\begin{prop}\label{splitting-refined}
\begin{align*}
\Omega (\rho_{\G}(M_1,E)\otimes \rho_{\G}(M_2,E))
= (-1)^{\textup{sign}} \tau (\mathcal{H})^2\rho_{\G}(M_1\#M_2,E).
\end{align*}
\end{prop}\ \\
\\[-7mm] This is an important result in the derivation of the actual gluing formula for refined analytic torsion and the final outcome of the preceeding three sections on cohomological algebra.

\section{Combinatorial complexes}
Before we finally prove a gluing formula for refined analytic torsion, consider a general situation with $Z$ being a smooth compact manifold and $Y\subset Z$ a smooth compact submanifold with the natural inclusion $\iota: Y \hookrightarrow Z$. The inclusion induces a group homomorphism $$\iota^*:\pi_1(Y) \to \pi_1(Z).$$
Fix any representation $\rho:\pi_1(Z)\to GL(n,\C)$. It naturally gives rise to further two representations
\begin{align*}
&\rho_Y:=\rho\circ \iota^*:\pi_1(Y)\to GL(n,\C), \\
&\bar{\rho}_Y:\frac{\pi_1(Y)}{\ker \iota^*}\cong im \, \iota^* \to GL(n,\C),\quad
[\gamma] \mapsto \rho_Y(\gamma),
\end{align*}
where the second map is well-defined since by construction $\rho_Y\restriction \ker \iota^*\equiv id$.
\\[3mm] Denote by $\widetilde{Z}$ and $\widetilde{Y}$ the universal covering spaces of $Z$ and $Y$ respectively, which are (cf. [KN, Proposition 5.9 (2)]) principal bundles over $Z,Y$ with respective structure groups $\pi_1(Z),\pi_1(Y)$. Denote by $p_Z$ the bundle projection of the principal bundle $\widetilde{Z}$ over $Z$. By locality the covering space $p^{-1}_Z(Y)$ over $Y$ is a principal bundle over $Y$ with the structure group $im \, \iota^*$. Hence the universal cover $\widetilde{Y}$ is a principal bundle over $p^{-1}_Z(Y)$ with the structure group $\ker \iota^*$. Summarizing we have: 
\begin{align}\label{covering1}
\frac{p^{-1}_Z(Y)}{im \, \iota^*} \cong Y, \quad \frac{\widetilde{Y}}{\ker \iota^*}\cong p^{-1}_Z(Y).
\end{align}
Next we consider any triangulation $\mathcal{Z}$ of $Z$, such that it leaves $Y$ invariant, i.e. $\mathcal{Y}:=\mathcal{Z}\cap Y$ provides a triangulation of the submanifold $Y$. Fix an embedding of $Z$ into $\widetilde{Z}$ as the fundamental domain. Then we obtain a triangulation $\mathcal{\widetilde{Z}}$ of $\widetilde{Z}$ by applying deck transformations of $\pi_1(Z)$ to $\mathcal{Z}$. Put
$$p^{-1}_Z(\mathcal{Y}):=\mathcal{\widetilde{Z}}\cap p^{-1}_Z(Y)$$
which gives a triangulation of $p^{-1}_Z(Y)$ invariant under deck transformations of $im \, \iota^*$. Embed $p^{-1}_Z(Y)$ into its universal cover $\widetilde{Y}$ as the fundamental domain. By applying deck transformations of $\ker \iota^*$ to $p^{-1}_Z(\mathcal{Y})$ we get a triangulation $\mathcal{\widetilde{Y}}$ of $\widetilde{Y}$. Note by construction, in analogy to \eqref{covering1}
\begin{align}\label{covering2}
\frac{p^{-1}_Z(\mathcal{Y})}{im \, \iota^*}\cong \mathcal{Y} , \quad \frac{\mathcal{\widetilde{Y}}}{\ker \iota^*}\cong p^{-1}_Z(\mathcal{Y}).
\end{align}
We form now the combinatorial chain complexes $C_*(.)$ of the triangulations and arrive at the central result of this section.
\begin{thm}\label{comb-complex}Consider the following combinatorial cochain complexes 
\begin{align*}
\textup{Hom}_{\rho_Y}(C_*&(\mathcal{\widetilde{Y}}),\C^n):=\{f \in\textup{Hom}(C_*(\mathcal{\widetilde{Y}}),\C^n) | \\ \forall x& \in C_*(\mathcal{\widetilde{Y}}), \gamma \in \pi_1(Y): f(x\cdot \gamma )=\rho_Y(\gamma)^{-1}f(x)\}, \\
\textup{Hom}_{\bar{\rho}_Y}(C_*&(p^{-1}_Z(\mathcal{Y})),\C^n):=\{f \in\textup{Hom}(C_*(p^{-1}_Z(\mathcal{Y})),\C^n) | \\ \forall x& \in C_*(p^{-1}_Z(\mathcal{Y})), \gamma \in im \, \iota^*: f(x\cdot \gamma )=\bar{\rho}_Y(\gamma)^{-1}f(x)\}.
\end{align*}
These complexes are isomorphic:
$$\textup{Hom}_{\rho_Y}(C_*(\mathcal{\widetilde{Y}}),\C^n) \cong \textup{Hom}_{\bar{\rho}_Y}(C_*(p^{-1}_Z(\mathcal{Y})),\C^n).$$
\end{thm}
\begin{proof}The relations in \eqref{covering2} imply in particular $$C_*(p^{-1}_Z(\mathcal{Y}))\cong C_*(\mathcal{\widetilde{Y}})/ \ker \iota^*.$$
Hence to any $x \in C_*(\mathcal{\widetilde{Y}})$ we can associate its equivalence class $[x]\in C_*(p^{-1}_Z(\mathcal{Y}))$ and define
\begin{align*}
\phi: \textup{Hom}_{\rho_Y}(C_*(\mathcal{\widetilde{Y}}),\C^n) &\rightarrow \textup{Hom}(C_*(p^{-1}_Z(\mathcal{Y})),\C^n), \\
f &\mapsto \phi f, \quad \phi f[x]:=f(x).
\end{align*}
This construction is well-defined, since for any other representative $x'\in [x]$ there exists $\gamma \in \ker \iota^*$ with $x'=x\cdot \gamma$ and since $f \in \textup{Hom}_{\rho_Y}(C_*(\mathcal{\widetilde{Y}}),\C^n)$ we get $$f(x')=\rho_Y(\gamma)^{-1}f(x)=[\rho \circ \iota^*(\gamma)]^{-1}f(x)=f(x).$$
Note further with $f$ and $x$ as above and $[\gamma]\in \pi_1(Y)/\ker \iota^* \cong im \, \iota^*$:
\begin{align*}
(\phi f)([x]\cdot [\gamma])=(\phi f)[x \cdot \gamma]=f(x\cdot \gamma)=\\=\rho_Y(\gamma)^{-1}f(x)=\bar{\rho}_Y([\gamma])^{-1}(\phi f)[x].
\end{align*}
Hence in fact we have a well-defined map:
$$\phi: \textup{Hom}_{\rho_Y}(C_*(\mathcal{\widetilde{Y}}),\C^n) \rightarrow \textup{Hom}_{\bar{\rho}_Y}(C_*(p^{-1}_Z(\mathcal{Y})),\C^n).$$
Now we denote the boundary operators on $C_*(\mathcal{\widetilde{Y}})$ and $C_*(p^{-1}_Z(\mathcal{Y}))$ by $\widetilde{\delta}$ and $\delta$ respectively. They give rise to coboundary operators $\widetilde{d}$ and $d$ on the cochain complexes. Observe
\begin{align*}
d(\phi f)[x]=(\phi f)(\delta [x])=(\phi f)[\widetilde{\delta} x]=\\=f(\widetilde{\delta} x)=\widetilde{d}f(x)=(\phi \widetilde{d}f)[x].
\end{align*}
This shows $$d \phi = \phi \widetilde{d}.$$
Thus $\phi$ is a well-defined homomorphism of complexes. It is surjective by construction. Injectivity of $\phi$ is also obvious. Thus $\phi$ is an isomorphism of complexes, as desired.
\end{proof}

\section{Gluing formula for Refined Analytic Torsion}\label{gluing-proof} 
We now finally are in the position to derive a gluing formula for refined analytic torsion. As a byproduct we obtain a splitting formula for the scalar analytic torsion in terms of combinatorial torsions of long exact sequences. 
\\[3mm] We derive the gluing formula by relating the Ray-Singer analytic torsion norm to the Reidemeister combinatorial torsion norm and applying the gluing formula on the combinatorial side, established by M. Lesch in [L2]. This makes it necessary to use the Cheeger-M\"{u}ller Theorem on manifolds with and without boundary.
\\[3mm] Continue in the setup of Section \ref{gluing-statement}. Consider a smooth triangulation $X$ of the closed smooth split manifold $$M=M_1\cup_N M_2$$ that leaves the compact submanifolds $M_1,M_2,N$ invariant, i.e. with $X_j:=X\cap M_j, j=1,2$ and $W:=X_j\cap N$ we have subcomplexes of $X$ providing smooth triangulations of $M_j,j=1,2$ and $N$ respectively, and $$X=X_1\cup_W X_2.$$
Denote by $\widetilde{M}_j$ the universal covering spaces of $M_j,j=1,2$. Fix embeddings of $M_j$ into $\widetilde{M}_j$ as fundamental domains. Then the triangulation $X_j$ of $M_j$ induces under the action of $\pi_1(M_j)$, viewed as the group of deck-transformations of $\widetilde{M}_j$, smooth triangulation $\widetilde{X}_j$ of the universal cover $\widetilde{M}_j$ for each $j=1,2$. 
\\[3mm] The complex chain group $C_*(\widetilde{X}_j)$ is generated by simplices of $\widetilde{X}_j$ and is a module over the group algebra $\C[\pi_1(X_j)]$. The simplices of $X_j$ form a preferred base for $C_*(\widetilde{X}_j)$ as a $\C[\pi_1(X_j)]$-module. 
\\[3mm] Furthermore the given unitary representation $\rho:\pi_1(M)\to U(n,\C)$ gives rise to the associated unitary representations $\rho_j:=\rho \circ \iota_j^*$ of the fundamental groups $\pi_1(M_j)$, where $\iota^*_j:\pi_1(M_j)\to \pi_1(M)$ are the natural group homomorphisms induced by the inclusions $\iota_j:M_j\hookrightarrow M, j=1,2$. We can now define for each $j$
\begin{align*}
C^*(X_j,\rho_j):=\textup{Hom}_{\rho_j}(C_*(\widetilde{X}_j),\C^n)=\hspace{70mm}\\
\{f \in \textup{Hom}(C_*(\widetilde{X}_j),\C^n)| \forall x \in C_*(\widetilde{X}_j), \gamma \in \pi_1(M_j): \ f(x\cdot \gamma)=\rho_j(\gamma)^{-1}f(x)\}\\
\cong C^*(\widetilde{X}_j) \otimes_{\C[\pi_1(M_j)]}\C^n,
\end{align*}
where the $\C[\pi_1(M_j)]$-module structure of $C^*(\widetilde{X}_j)$ comes from the module structure of the dual space $C_*(\widetilde{X}_j)$ and the $\C[\pi_1(M_j)]$ module structure on $\C^n$ is obtained via the representation $\rho_j$.
\\[3mm] The boundary operator on $C_*(\widetilde{X}_j)$ induces a coboundary operator on $C^*(X_j,\rho_j)$. Further the preferred base on $C_*(\widetilde{X}_j)$ together with a fixed volume on $\C^n$ yields a Hilbert structure on $C^*(X_j,\rho_j)$. So $C^*(X_j,\rho_j)$ becomes a finite Hilbert complex.
\\[3mm] Next we consider again the universal coverings $p_j:\widetilde{M}_j\to M_j$. Then the preimage $p_j^{-1}(N)\subset \widetilde{M}_j$ is a covering space of $N$ with the group of deck transformations $$\textup{im}(\pi_1(N)\xrightarrow{\iota^*} \pi_1(M_j))\subset \pi_1(M_j).$$ Here $\iota^*$ is the natural homomorphism of groups induced by the inclusion $N\hookrightarrow M_j$. We do not distinguish the inclusions of $N$ into $M_j,j=1,2$ at this point, since it will always be clear from the context.
\\[3mm] The triangulation $W\subset X_j$ induces with a fixed embedding of $M_j$ into $\widetilde{M}_j$ a triangulation $p^{-1}_j(W)$ of $p_j^{-1}(N)$ by the deck tranformations of $\iota^*\pi_1(N)\subset \pi_1(M_j)$. The chain complex $C_*(p^{-1}_j(W))$ is generated by simplices of $p^{-1}_j(W)$ and is a module over the group subalgebra $\C [\iota^*\pi_1(N)]\subset \C[\pi_1(M_j)]$. It is a subcomplex of $C_*(\widetilde{X}_j)$. 
\\[3mm] The following observation follows from Theorem \ref{comb-complex} and is central for the later constructions:
\begin{align}\nonumber
&f\in \textup{Hom}_{\rho_j}(C_*(\widetilde{X}_j),\C^n) \\ \label{isomorphism} \Rightarrow &f|_{C_*(p^{-1}_j(W))} \in \textup{Hom}_{\bar{\rho}_N}(C_*(p^{-1}_j(W)),\C^n)\cong C^*(W,\rho_N),
\end{align}
where $\rho_N=\rho_j\circ \iota^*: \pi_1(N)\to U(n,\C)$ is the natural representation of $\pi_1(N)$. It is induced by $\rho$ and is trivial over $\ker \, \iota^*$ by definition. The homomorphism $\bar{\rho}_N$ is obtained from $\rho_N$ by dividing out the trivial part:
$$\bar{\rho}_N:\frac{\pi_1(N)}{\ker \, \iota^*}\cong im \, \iota^* \to U(n,\C).$$
The isomorphism in \eqref{isomorphism} in particular implies that we can compare the restrictions to $C_*(p^{-1}_j(W))$ for elements of both complexes $C^*(X_j,\rho_j),j=1,2$. We can now define:
\begin{align*}
&C^*(X_j,W,\rho_j):=\{f\in \textup{Hom}_{\rho}(C_*(\widetilde{X}_j),\C^n) | f|_{C_*(p^{-1}_j(W))}=0\}, \\
&C^*(X_1\#X_2,\rho):=\{(f,g)\in C^*(X_1,\rho_1)\oplus C^*(X_2,\rho_2) | f|_{C_*(p^{-1}_j(W))}=g|_{C_*(p^{-1}_j(W))}\}. 
\end{align*}
These complexes inherit structure of finite Hilbert complexes from $C^*(X_j,\rho_j)$ for $j=1,2$. Fix the naturally induced Hilbert structure on the cohomology, which gives rise to norms on the determinant lines of cohomology, and define the combinatorial Reidemeister norms
\begin{align*}
&\|\cdot \|^R_{\det H^*(C^*(X_1\#X_2,\rho))}:= \tau (C^*(X_1\#X_2,\rho))^{-1} \|\cdot \|_{\det H^*(C^*(X_1\#X_2,\rho))},\\ 
&\|\cdot \|^R_{\det H^*(C^*(X_j,W,\rho))}:= \tau (C^*(X_j,W,\rho))^{-1}\|\cdot \|_{\det H^*(C^*(X_j,W,\rho))},\\
&\|\cdot \|^R_{\det H^*(C^*(X_j,\rho))}:=\tau (C^*(X_j,\rho))^{-1}\|\cdot \|_{\det H^*(C^*(X_j,\rho))},
\end{align*}
where we have put for any finite Hilbert complex $(C^*,\partial_*)$ with the naturally induced Hilbert structure on cohomology $H^*(C^*,\partial_*)$ and the associated Laplacians denoted by $\triangle_*$: $$\log \tau (C^*,\partial_*)=\frac{1}{2}\sum_{j}(-1)^j\cdot j\cdot \zeta'(0,\triangle_j).$$
This definition corresponds to the sign convention for the Ray-Singer norms in [BV3, Section 5]. The Reidemeister norms do not depend on choices made for the construction and are in particular invariant under subdivisions, see [Mi, Theorem 7.1] and [RS, Section 4]. Since any two smooth triangulations admit a common subdivision, see [Mun], the Reidemeister norms do not depend on the choice of a smooth triangulation $X$. 
\\[3mm] Consider now the following short exact sequences of finite Hilbert complexes:
\begin{align}\label{A-complex}
0\rightarrow C^*(X_1,W,\rho)\xrightarrow{\A_c}C^*(X_1\#X_2,\rho)\xrightarrow{\beta_c}C^*(X_2,\rho)\rightarrow 0, \\
\label{B-complex}
0\rightarrow C^*(X_2,W,\rho)\xrightarrow{\A'_c}C^*(X_1\#X_2,\rho)\xrightarrow{\beta'_c}C^*(X_1,\rho)\rightarrow 0,
\end{align}
where $\A_c,\A'_c$ are the natural inclusions and $\beta_c,\beta'_c$ the natural restrictions. Both sequences are exact by definition of the corresponding homomorphisms of complexes. The associated long exact sequences in cohomology, with the Hilbert structures being naturally induced by the Hilbert structures of the combinatorial complexes as defined above, shall be denoted by $\mathcal{H}_c$ and $\mathcal{H'}_c$ respectively.
\\[3mm] Consider further the following complexes
\begin{align*}
0 \rightarrow (\Omega^*_{\min}(M_1,E),\D_1)\xrightarrow{\A} (\Omega^*(M_1\#M_2,E),\D_S) \xrightarrow{\beta} (\Omega^*_{\max}(M_2,E),\D_2)\rightarrow 0, \\
0 \rightarrow (\Omega^*_{\min}(M_2,E),\D_2)\xrightarrow{\A'} (\Omega^*(M_1\#M_2,E),\D_S) \xrightarrow{\beta'} (\Omega^*_{\max}(M_1,E),\D_1)\rightarrow 0,
\end{align*}
which were already introduced in Section \ref{cohom-algebra}. Their associated long exact sequences (cf. \eqref{LES-H}) are denoted by $\mathcal{H}$ and $\mathcal{H'}$ respectively. The short exact sequences commute under the de Rham maps with the short exact sequences \eqref{A-complex} and \eqref{B-complex},
respectively. 
\\[3mm] Thus the corresponding diagramms of the long exact sequences $\mathcal{H},\mathcal{H}_c$ and $\mathcal{H}',\mathcal{H}'_c$ commute. The de Rham maps induce isomorphisms on cohomology, as established in [RS, Section 4] with arguments for orthogonal representations which work for unitary representations as well:
\begin{align}\label{deRham1}
H^*_{\textup{abs}}(M_j,E)\cong H^*(C^*(X_j,\rho_j)), \ H^*_{\textup{rel}}(M_j,E)\cong H^*(C^*(X_j,W,\rho_j)).
\end{align}
>From these identifications we obtain with the five-lemma in algebra applied to the commutative diagramms of long exact sequences $\mathcal{H}$ and $\mathcal{H}_c$ or $\mathcal{H}'$ and $\mathcal{H}'_c$:
\begin{align}\label{deRham3}
H^*(M_1\#M_2,E)\cong H^*(C^*(X_1\#X_2,\rho)), 
\end{align}
induced by the de Rham integration maps as well. Thus under the de Rham isomorphisms the long exact sequences $\mathcal{H}_c,\mathcal{H}'_c$ correspond to $\mathcal{H},\mathcal{H}'$ respectively, and differ only in the fixed Hilbert structures.
\\[3mm] Furthermore the long exact sequences $\mathcal{H}_c,\mathcal{H}'_c$ give rise to isomorphisms on determinant lines in a canonical way (recall the definition of $\Psi, \Psi'$ in \eqref{psi1} and \eqref{psi2})
\begin{align*}
\Psi_c:\det H^*(C^*(X_1,W,\rho))\otimes \det H^*(C^*(X_2,\rho)) \rightarrow \det H^*(C^*(X_1\#X_2,\rho)), \\
\Psi'_c:\det H^*(C^*(X_2,W,\rho))\otimes \det H^*(C^*(X_1,\rho)) \rightarrow \det H^*(C^*(X_1\#X_2,\rho)).
\end{align*}
These maps correspond to the canonical identifications $\Psi,\Psi'$ introduced in Section \ref{canonical} up to the de Rham isomorphisms. We can now prove an appropriate gluing result for the combinatorial Reidemeister norms.
\begin{thm}\label{comb-gluing} \ \\
Let $x,y$ be elements of $\det H^*(C^*(X_1,W,\rho)),\det H^*(C^*(X_2,\rho))$ and $x',y'$ elements of $\det H^*(C^*(X_2,W,\rho)), \det H^*(C^*(X_1,\rho))$, respectively. Then we obtain for the combinatorial Reidemeister norms the following relation:
\begin{align}
&\|\Psi_c(x\otimes y)\|^R_{\det H^*(C^*(X_1\#X_2,\rho))}= \label{r-norms1} \\ = &2^{\chi(N)/2}\|x\|^R_{\det H^*(C^*(X_1,W,\rho))}\|y\|^R_{\det H^*(C^*(X_2,\rho))}, \nonumber \\
&\|\Psi'_c(x'\otimes y')\|^R_{\det H^*(C^*(X_1\#X_2,\rho))}= \label{r-norms2} \\ = &2^{\chi(N)/2}
\|x'\|^R_{\det H^*(C^*(X_2,W,\rho))}\|y'\|^R_{\det H^*(C^*(X_1,\rho))}. \nonumber 
\end{align}
\end{thm}
\begin{proof}
First apply the gluing formula in [L2], derived by introducing transmission boundary conditions depending on a parameter, in the spirit of [V]:
\begin{align}
\tau (C^*(X_1\#X_2,\rho))=\tau (C^*(X_1,W,\rho)) \cdot \tau (C^*(X_2,\rho))\cdot \tau (\mathcal{H}_c)\cdot 2^{-\chi(N)/2}, \label{lesch1} \\
\tau (C^*(X_1\#X_2,\rho))=\tau (C^*(X_2,W,\rho)) \cdot \tau (C^*(X_1,\rho))\cdot \tau (\mathcal{H'}_c)\cdot 2^{-\chi(N)/2}. \label{lesch2}
\end{align}
By the definition of the combinatorial torsions $\tau(\mathcal{H}_c)$ and $\tau(\mathcal{H'}_c)$ we obtain the following relation to the action of $\Phi_c,\Phi'_c$, by an appropriate version of Corollary \ref{phi-tau-2}:
\begin{align}
&\|\Psi_c(x\otimes y)\|_{\det H^*(C^*(X_1\#X_2,\rho))}= \label{comb-norms1} \\ = &\tau (\mathcal{H}_c)\cdot \|x\|_{\det H^*(C^*(X_1,W,\rho))}\|y\|_{\det H^*(C^*(X_2,\rho))}, \nonumber \\
&\|\Psi'_c(x'\otimes y')\|_{\det H^*(C^*(X_1\#X_2,\rho))}= \label{comb-norms2} \\ = &\tau (\mathcal{H'}_c)\cdot
\|x'\|_{\det H^*(C^*(X_2,W,\rho))}\|y'\|_{\det H^*(C^*(X_1,\rho))}. \nonumber 
\end{align}
Now a combination of the relations above, together with the gluing formulas \eqref{lesch1} and \eqref{lesch2} gives the desired statement.
\end{proof}\ \\
\\[-7mm] We can now prove the following gluing result for the analytic Ray-Singer torsion norms.

\begin{thm}\label{rs-gluing}
Let $\Omega$ be the canonical isomorphism of determinant lines, defined in \eqref{omega}.
\begin{align*}
\Omega : \det H^*(\domm_1,\DD_1) \otimes \det H^*(\domm_2,\DD_2)\rightarrow \hspace{30mm} \\ \det (H^*(M_1\#M_2,E)\oplus H^*(M_1\#M_2,E)).
\end{align*}
For any $\gamma_1,\gamma_2$ in $\det H^*(\domm_1,\DD_1), \det H^*(\domm_2,\DD_2)$ respectively, we have in terms of the analytic Ray-Singer torsion norms on the determinant lines:
\begin{align*}
&\|\Omega (\gamma_1\otimes \gamma_2)\|^{RS}_{\det (H^*(M_1\#M_2,E)\oplus H^*(M_1\#M_2,E))}= \\=&
2^{\chi(N)}\|\gamma_1\|^{RS}_{\det H^*(\domm_1,\DD_1)}\|\gamma_2\|^{RS}_{\det H^*(\domm_2,\DD_2)}.
\end{align*}
\end{thm}
\begin{proof}
Under the de Rham isomorphisms we can relate the combinatorial Reidemeister norms to the analytic Ray-Singer torsion norms. We get by an appropriate version of [L\"{u}] 
\begin{align}
&\|\cdot \|^R_{\det H^*(C^*(X_j,\rho))}=2^{\chi(N)/4}\|\cdot \|^{RS}_{\det H^*_{\textup{abs}}(M_j,E)}, \nonumber \\ 
&\vspace{20mm} \|\cdot \|^R_{\det H^*(C^*(X_j,W, \rho))}=2^{\chi(N)/4}\|\cdot \|^{RS}_{\det H^*_{\textup{rel}}(M_j,E)}, \label{lueck}
\end{align}
where $\chi(N)$ is the Euler characteristic of the closed manifold $N$ with the representation $\rho_N$ of its fundamental group, hence defined in terms of the twisted cohomology groups $H^*(N,E|_N)$.
Furthermore we need the following relation:
\begin{align}\label{split-cheeger}
\|\cdot \|^R_{\det H^*(C^*(X_1\#X_2,\rho))}=2^{\chi(N)/2}\|\cdot \|^{RS}_{\det H^*(M_1\#M_2,E)}.
\end{align}
This result is proved for trivial representations in [V, Theorem 1.5]. This is done by discussing a family of elliptic transmission value problems and doesn't rely on the Cheeger-M\"{u}ller theorem. However in the setup of the present discussion we provide below in Proposition \ref{anticipation} a simple proof for general unitary representations, using the Cheeger-M\"{u}ller Theorem on closed manifolds
\\[3mm] It is important to note that the Ray-Singer analytic and combinatorial torsion considered in [V] and [L\"{u}] are squares of the torsion norms in our convention and further differ in the sign convention (we adopted the sign convention of [BK2, Section 11.2]). Therefore we get factors $2^{\chi(N)/4}, 2^{\chi(N)/2}$ in \eqref{lueck} and \eqref{split-cheeger} respectively, instead of $2^{-\chi(N)/2}, 2^{-\chi(N)}$ as asserted in [L\"{u}, Theorem 4.5] and [V, Theorem 1.5].
\\[3mm] By definition the canonical maps $\Psi_c$ and $\Psi'_c$ correspond under the de Rham isomorphism to the canonical maps $\Psi$ and $\Psi'$ respectively. In view of Theorem \ref{comb-gluing}, the identities \eqref{lueck} and the relation \eqref{split-cheeger} we obtain the following gluing formulas:
\begin{align}
&\|\Psi(x\otimes y)\|^{RS}_{\det H^*(M_1\#M_2,E))}= \label{a-norms1} \\ = &2^{\chi(N)/2}\|x\|^{RS}_{\det H^*_{\textup{rel}}(M_1,E))}\|y\|^{RS}_{\det H^*_{\textup{abs}}(M_2,E))}, \nonumber \\
&\|\Psi'(x'\otimes y')\|^{RS}_{\det H^*(M_1\#M_2,E))}= \label{a-norms2} \\ = 
&2^{\chi(N)/2}\|x'\|^{RS}_{\det H^*_{\textup{rel}}(M_2,E))}\|y'\|^{RS}_{\det H^*_{\textup{abs}}(M_1,E))}. \nonumber 
\end{align}
The fusion isomorphisms $\mu_{(M_1,E)},\mu_{(M_2,E)}$ and $\mu_{(M_1\#M_2,E)}$, used in the construction of the canonical isomorphism $\Omega$, are by construction isometries with respect to the analytic Ray-Singer norms and hence in total we obtain for any $\gamma_1,\gamma_2$ in $\det H^*(\domm_1,\DD_1), \det H^*(\domm_2,\DD_2)$ respectively, 
\begin{align*}
\|\Omega (\gamma_1\otimes \gamma_2)\|^{RS}_{\det (H^*(M_1\#M_2,E)\oplus H^*(M_1\#M_2,E)}= \\=
2^{\chi(N)}\|\gamma_1\|^{RS}_{\det H^*(\domm_1,\DD_1)}\|\gamma_2\|^{RS}_{\det H^*(\domm_2,\DD_2)},
\end{align*}
where we recall the following facts by construction:
\begin{align*}
H^*(\domm_1,\DD_1)=H^*_{\textup{rel}}(M_1,E)\oplus H^*_{\textup{abs}}(M_1,E), \\
H^*(\domm_2,\DD_2)=H^*_{\textup{rel}}(M_2,E)\oplus H^*_{\textup{abs}}(M_2,E).
\end{align*}
This proves the statement of the theorem.
\end{proof}\ \\
\\[-7mm] Now we prove the result \eqref{split-cheeger} on comparison of the torsion norms, anticipated in the argumentation above. The proof uses ideas behind [V, Theorem 1.5].
\begin{prop}\label{anticipation}
\begin{align*}
\|\cdot \|^R_{\det H^*(C^*(X_1\#X_2,\rho))}=2^{\chi(N)/2}\|\cdot \|^{RS}_{\det H^*(M_1\#M_2,E)}. 
\end{align*}
\end{prop}
\begin{proof}
Consider the following short exact sequence of finite Hilbert complexes (recall that the Hilbert structures on the complexes were induced by the triangulation $X$ and the fixed volume on $\C^n$)
\begin{align*}
0 \to \bigoplus\limits_{j=1}^2C^*(X_j,W,\rho_j)\xrightarrow{\A}C^*(X_1\#X_2,\rho)\xrightarrow{\beta}C^*(W,\rho_N)\to 0,
\end{align*}
with $\A(\w_1\oplus \w_2)=(\w_1,\w_2)$ and $\beta(\w_1,\w_2)=\frac{1}{\sqrt{2}}(\w_1|_{C_*(\widetilde{W})}+\w_2|_{C_*(\widetilde{W})})$. Note further 
\begin{align*}
\theta: C^*(W,\rho_N)\to C^*(X_1\#X_2,\rho), \ \theta(\w)=\frac{1}{\sqrt{2}}(\w,\w)
\end{align*} 
is an isometry between $C^*(W,\rho_N)$ and Im$\theta$, where Im$\theta$ is moreover the orthogonal complement in $C^*(X_1\#X_2,\rho)$ to the image of $\A$. Furthermore $\beta \circ \theta=id$. Hence $\beta$ is an isometry between the orthogonal complement of its kernel and $C^*(W,\rho_N)$. Here a volume on $\C^n$ is fixed for all combinatorial complexes. 
\\[3mm] The map of complexes $\A$ is also an isometry onto its image and hence the induced identification $$\phi^R_{\#}:\det H^*(C^*(X_1\#X_2,\rho))\to \bigotimes\limits_{j=1}^2\det H^*(C^*(X_j,W,\rho_j))\otimes \det H^*(C^*(W,\rho_N))$$ is an isometry of combinatorial Reidemeister norms. Similarly we consider the next short exact sequence of finite complexes:
\begin{align*}
0 \to \bigoplus\limits_{j=1}^2C^*(X_j,W,\rho_j)\xrightarrow{\A}C^*(X,\rho)\xrightarrow{r} C^*(W,\rho_N)\to 0,
\end{align*}
where the third arow is the restriction as in \eqref{isomorphism}. By similar arguments as before the induced identification 
$$\phi^R: \det H^*(C^*(X,\rho))\to \bigotimes\limits_{j=1}^2\det H^*(C^*(X_j,W,\rho_j))\otimes \det H^*(C^*(W,\rho_N))$$ is an isometry of combinatorial Reidemeister norms. Now we consider the following commutative diagramms of short exact sequences.
\begin{align*}
\begin{array}{ccccccccc}
0 & \to & \bigoplus_{j=1}^2\Omega^*_{\textup{min}}(M_j,E) & \hookrightarrow & \Omega^*(M_1\#M_2,E) & \xrightarrow{\sqrt{2}\iota_N^*}&\Omega^*(N,E) & \to & 0 \\
 & & \downarrow R & & \downarrow R & & \downarrow R & & \\
0 & \to & \bigoplus\limits_{j=1}^2C^*(X_j,W,\rho_j)& \xrightarrow{\A}& C^*(X,\rho)& \xrightarrow{\sqrt{2}r} & C^*(W,\rho_N) & \to & 0,
\end{array}
\end{align*}
where $R$ denotes the natural de Rham integration quasi-isomorphisms. The second commutative diagramm is as follows:
\begin{align*}
\begin{array}{ccccccccc}
0 & \to & \bigoplus_{j=1}^2\Omega^*_{\textup{min}}(M_j,E) & \hookrightarrow & \Omega^*(M_1\#M_2,E) & \xrightarrow{\sqrt{2}\iota_N^*}&\Omega^*(N,E) & \to & 0 \\
 & & \downarrow R & & \downarrow R_{\#} & & \downarrow R & & \\
0 & \to & \bigoplus\limits_{j=1}^2C^*(X_j,W,\rho_j)& \xrightarrow{\A}& C^*(X_1\#X_2,\rho)& \xrightarrow{\beta} & C^*(W,\rho_N) & \to & 0,
\end{array}
\end{align*}
with the vertical maps as before given by the natural de Rham integration quasi-isomorphisms. Note that $R_{\#}$ is a quasi-isomorphism as well, which is clear from the five-lemma applied to the commutative diagramm of the associated long exact sequences.
\\[3mm] The lower sequences in both of the diagramms were discussed above. The upper short exact sequence in both diagramms induces the identification:
$$\phi^{RS}_{\#}:\det H^*(M_1\#M_2,E)\to \bigotimes\limits_{j=1}^2\det H^*_{\textup{rel}}(M_j,E)\otimes \det H^*(N,E).$$
By commutativity of the two diagramms we obtain:
\begin{align}\label{comparison1}
2^{\chi(N)/2}\phi^R\circ R&=R \circ \phi^{RS}_{\#}, \\ \label{comparison2}
\phi^R_{\#}\circ R_{\#}&=R\circ \phi^{RS}_{\#}.
\end{align}
Now let $x\in \det H^*(M_1\#M_2,E)$ be an arbitrary element, identified via Corollary \ref{split-cohomology} with $x\in \det H^*(M,E)$. We compute:
\begin{align*}
\|x\|^{RS}_{\det H^*(M_1\#M_2,E)}=\|x\|^{RS}_{\det H^*(M,E)}=\|R(x)\|^{R}_{\det H^*(X,\rho)}=\\ = \|\phi^R \circ R(x)\|= 2^{-\chi(N)/2}\|R \circ \phi^{RS}_{\#}(x)\|= \\ =
2^{-\chi(N)/2}\|\phi^R_{\#}\circ R_{\#}(x)\|=2^{-\chi(N)/2}\|R_{\#}(x)\|^{R}_{\det H^*(C^*(X_1\#X_2,\rho))},
\end{align*}
where we have put
$$\|\cdot \| := \|\cdot\|^R_{\det H^*(C^*(X_1,W,\rho_1))} \cdot \|\cdot \|^R_{\det H^*(C^*(X_2,W,\rho_2))} \cdot \|\cdot \|^R_{\det H^*(C^*(W,\rho_N))}.$$
The steps in the sequence of equalities need to be clarified. The first equation is due to Theorem \ref{split-laplacian} on the spectral equivalence of $\triangle_S$ and $\triangle$. The second equation is simply the Cheeger-M\"{u}ller Theorem for closed Riemannian manifolds. The third equation is a consequence of the fact that $\phi^R$ is an isometry with respect to the combinatorial Reidemeister norms. 
Now the fourth and the fifth equation are consequences of \eqref{comparison1} and \eqref{comparison2} respectively. Using in the last equation again the isometry $\phi^R_{\#}$ we obtain the result. The sequence of equalities proves in total:
$$\|x\|^{RS}_{\det H^*(M_1\#M_2,E)}=2^{-\chi(N)/2}\|R_{\#}(x)\|^{R}_{\det H^*(C^*(X_1\#X_2,\rho))}.$$ 
\end{proof} \ \\
\\[-7mm] Next we recall that by Theorem \ref{split-laplacian} the complexes $\Omega^*(M,E)\oplus \Omega^*(M,E)$ and $\Omega^*(M_1\#M_2,E)\oplus \Omega^*(M_1\#M_2,E)$ have spectrally equivalent Laplacians with identifyable eigenforms. This implies
\begin{align}
T^{RS}(\Omega^*(M,E)\oplus \Omega^*(M,E))=T^{RS}(\Omega^*(M_1\#M_2,E)\oplus \Omega^*(M_1\#M_2,E)),\nonumber \\ \label{split-cohomology2}
H^*(\Omega^*(M,E)\oplus \Omega^*(M,E))\cong
H^*(\Omega^*(M_1\#M_2,E)\oplus \Omega^*(M_1\#M_2,E)).
\end{align}
The identification \eqref{split-cohomology2} is in fact an isometry with respect to the natural Hilbert structures, since in both cases the Hilbert structure is induced by the $L^2-$scalar product on harmonic forms and the harmonic forms of both complexes coincide, see Theorem \ref{split-laplacian}. This implies
\begin{align*}
&\|\cdot\|^{RS}_{\det (H^*(M_1\#M_2,E)\oplus H^*(M_1\#M_2,E))}= \\  &\|\cdot\|^{RS}_{\det (H^*(M,E)\oplus H^*(M,E))}\equiv \|\cdot \|^{RS}_{\det H^*(\domm, \DD)}.
\end{align*}
Moreover under the identification \eqref{split-cohomology2} we can view the canonical isomorphism $\Omega$ as
$$\Omega : \det H^*(\domm_1,\DD_1)\otimes \det H^*(\domm_2,\DD_2)\rightarrow \det H^*(\domm,\DD).$$
Then we obtain as a corollary of Theorem \ref{rs-gluing}:
\begin{cor}\label{phase}
Denote by $\rho_{\textup{an}}(\D)$ and $\rho_{\textup{an}}(\D_j),j=1,2$ the refined analytic torsions on $M$ and $M_j,j=1,2$ respectively. There exists some $\phi \in [0,2\pi)$ such that 
$$\Omega (\rho_{\textup{an}}(\D_1)\otimes \rho_{\textup{an}}(\D_2))=e^{i\phi}2^{\chi(N)}\rho_{\textup{an}}(\D).$$
\end{cor}
\begin{proof}
Applying Theorem \ref{rs-gluing} to $\rho_{\textup{an}}(\D_1)\otimes \rho_{\textup{an}}(\D_2)$, we obtain with the identification \eqref{split-cohomology2} and [BV3, Theorem 5.2]:
\begin{align*}
&\|\Omega (\rho_{\textup{an}}(\D_1)\otimes \rho_{\textup{an}}(\D_2))\|^{RS}_{\det (H^*(\domm,\DD)}=
2^{\chi(N)} \times \\ \times & \|\rho_{\textup{an}}(\D_1)\|^{RS}_{\det H^*(\domm_1,\DD_1)}\|\rho_{\textup{an}}(\D_2)\|^{RS}_{\det H^*(\domm_2,\DD_2)}= 2^{\chi(N)}.
\end{align*}
On the other hand we have again by [BV3, Theorem 5.2]
$$\|\rho_{\textup{an}}(\D)\|^{RS}_{\det (H^*(\domm,\DD)}=1.$$
This proves the corollary.
\end{proof}\ \\
\\[-7mm] In order to establish a gluing formula it remains to identify this phase $\phi$ explicitly. Under the identification of Corollary \ref{split-cohomology} the refined torsion $\rho_{\G}(M_1\#M_2,E)$ corresponds to the refined torsion element $\rho_{\G}(M,E)$, as already encountered in \eqref{split-refined2}. Hence with Proposition \ref{splitting-refined} we can write 
$$\Omega (\rho_{\G}(M_1,E)\otimes \rho_{\G}(M_2,E))
= (-1)^{\textup{sign}} \tau (\mathcal{H})^2\rho_{\G}(M,E).$$
Consequently we obtain using the splitting formulas \eqref{splitting-eta} and \eqref{splitting-eta-trivial} for the eta-invariants and using Proposition \ref{rho-formula}
\begin{align}
\Omega (\rho_{\textup{an}}(\D_1)\otimes \rho_{\textup{an}}(\D_2))=\frac{T^{RS}(\domm,\DD)}{T^{RS}(\domm_1,\DD_1)T^{RS}(\domm_2,\DD_2)}\times \nonumber \\
\exp \left(-i\pi \cdot \textup{Err}\eta(\B_{\textup{even}})+i\pi \cdot \textup{rank}(E)\textup{Err}\eta(\B_{\textup{trivial}})\right)\times \nonumber \\
(-1)^{\textup{sign}}\tau (\mathcal{H})^2 \cdot \rho_{\textup{an}}(\D),\label{relation}
\end{align}
where we have put 
\begin{align}\label{error1}
\textup{Err}\eta(\B_{\textup{even}}):=&\tau_{\mu}(I-P_1,P,P_1)+\\ \nonumber +&\frac{m-1}{2}\left(\dim \ker \B^1_{\textup{even}}+\dim \ker \B^2_{\textup{even}}-\dim \ker \B_{\textup{even}}\right), \\ \label{error2}
\textup{Err}\eta(\B_{\textup{trivial}}):=&\tau_{\mu}(I-P_{1, \textup{trivial}},P_{\textup{trivial}},P_{1, \textup{trivial}})+\\ \nonumber +&\frac{m}{2}\left(\dim \ker \B^1_{\textup{trivial}}+\dim \ker \B^2_{\textup{trivial}}-\dim \ker \B_{\textup{trivial}}\right).
\end{align}
Comparing this relation with the statement of Corollary \ref{phase} we obtain for the phase $\phi$ in Corollary \ref{phase}:
$$\phi=\pi \left[ -\textup{Err}\eta(\B_{\textup{even}})+ \textup{rank}(E)\textup{Err}\eta(\B_{\textup{trivial}})+\textup{sign}\right].$$
Due to the definition of the refined analytic torsion in [BV3, (5.17)], it makes sense to reduce the phase $\phi$ modulo $\pi \textup{rank}(E)\Z$. 
\begin{lemma}
$$\phi\equiv \pi \left[ \textup{sign}-\textup{Err}\eta(\B_{\textup{even}})\right]\ \textup{mod}\ \pi \textup{rank}(E)\Z.$$
\end{lemma}
\begin{proof}
We need to verify $$\textup{Err}\eta(\B_{\textup{trivial}})\equiv 0 \ \textup{mod}\ \Z.$$
Denote the Laplace operators of the complexes $(\Omega^*_{\textup{min/max}}(M_j), \D_{j,\textup{trivial}} )$ by $\triangle^j_{\textup{rel/abs}}$, respectively. Let $\triangle$ denote the Laplacian of the complex $(\Omega^*(M), \D_{\textup{trivial}} )$. We have by construction 
\begin{align*}
&\B(\D_{\textup{trivial}})^2=\triangle\oplus \triangle, \\
&\B(\D_{j, \textup{trivial}})^2=\triangle^j_{\textup{rel}}\oplus \triangle^j_{\textup{abs}}.
\end{align*}
Since the Maslov-triple index $\tau_{\mu}$ is integer-valued, we obtain via [BV3, Lemma 4.8] the following mod $\Z$ calculation:
\begin{align*}
\textup{Err}\eta(\B_{\textup{trivial}})=\frac{1}{2}&\sum_{k=0}^m(-1)^{k+1}\cdot k \cdot (\dim \ker \triangle^1_{k,\textup{rel}}+ \dim \ker \triangle^1_{k,\textup{abs}})+ \\ +
\frac{1}{2}&\sum_{k=0}^m(-1)^{k+1}\cdot k \cdot (\dim \ker \triangle^2_{k,\textup{rel}}+ \dim \ker \triangle^2_{k,\textup{abs}})- \\ -
\frac{1}{2}&\sum_{k=0}^m(-1)^{k+1}\cdot k \cdot (\dim \ker \triangle_{k}+ \dim \ker \triangle^1_{k}).
\end{align*}
The Poincare duality implies:
\begin{align*}
\dim \ker \triangle^j_{k,\textup{rel}}&= \dim \ker \triangle^j_{m-k,\textup{abs}}, \\
\dim \ker \triangle_{k}&= \dim \ker \triangle_{m-k}.
\end{align*}
Hence we compute further modulo $\Z$
\begin{align*}
\textup{Err}\eta(\B_{\textup{trivial}})=\frac{m}{2}\sum_{k=0}^m(-1)^k \dim \ker \triangle^1_{k,\textup{rel}}+ \\ +
\frac{m}{2}\sum_{k=0}^m(-1)^k \dim \ker \triangle^2_{k,\textup{abs}}- \frac{m}{2}\sum_{k=0}^m(-1)^k \dim \ker \triangle_{k}.
\end{align*}
Finally, exactness of the long exact sequence $\mathcal{H}$ in \eqref{LES-H} (in the setup of a trivial line bundle) implies $\textup{Err}\eta(\B_{\textup{trivial}})\equiv 0 \ \textup{mod} \ \Z$.
\end{proof} \ \\
\\[-7mm] We finally arrive at the following central result: a gluing formula for refined analytic torsion. 
\begin{thm}\label{gluing-final}\textup{[Gluing formula for Refined Analytic Torsion]}\\
Let $M=M_1\cup_{N}M_2$ be an odd-dimensional oriented closed Riemannian split-manifold where $M_j, j=1,2$ are compact bounded Riemannian manifolds with $\partial M_j=N$ and orientation induced from $M$. Denote by $(E,\D, h^E)$ a complex flat vector bundle induced by an unitary representation $\rho:\pi_1(M)\to U(n,\C)$. Assume product structure for the metrics and the vector bundle. Set:
\begin{align*}
&(\domm_j,\DD_j):=(\dom_{j,\min}, \D_{j, \min}) \oplus (\dom_{j,\max}, \D_{j, \max}), \ j=1,2, \\
&(\domm, \DD):=(\Omega^*(M,E),\D)\oplus (\Omega^*(M,E),\D).
\end{align*}
The canonical isomorphism
$$\Omega : \det H^*(\domm_1,\DD_1)\otimes \det H^*(\domm_2,\DD_2)\rightarrow \det H^*(\domm, \DD)$$ is induced by the long exact sequences on cohomologies 
\begin{align*} &\mathcal{H}: \ 
...H^k_{\textup{rel}}(M_1,E)\rightarrow H^k(M,E)\rightarrow H^k_{\textup{abs}}(M_2,E)\rightarrow H^{k+1}_{\textup{rel}}(M_1,E)... \\  &\mathcal{H'}\!: ...H^k_{\textup{rel}}(M_2,E)\rightarrow H^k(M,E)\rightarrow H^k_{\textup{abs}}(M_1,E)\rightarrow H^{k+1}_{\textup{rel}}(M_2,E)...\\
\end{align*}
and fusion isomorphisms. The isomorphism $\Omega$ is linear, hence well-defined on equivalence classes modulo multiplication by $\exp [i\pi \textup{rk}E]$. Then the gluing formula for refined analytic torsion in [BV3, (5.17)] is given as follows:
\begin{align*}
&\Omega (\rho_{\textup{an}}(M_1,E)\otimes \rho_{\textup{an}}(M_2,E))=K(M,M_1,M_2,\rho)\cdot\rho_{\textup{an}}(M,E), \\
&K(M,M_1,M_2,\rho):= 2^{\chi(N)}
\exp (i\phi),\\
&\phi:= \pi (\textup{sign} - \textup{Err}\eta(\B_{\textup{even}})).
\end{align*}
The term $\textup{Err}\eta(\B_{\textup{even}})$ is an error term in the gluing formula for eta-invariants 
\begin{align*}
\textup{Err}\eta(\B_{\textup{even}}):=&\tau_{\mu}(I-P_1,P,P_1)+\\  +&\frac{m-1}{2}\left(\dim \ker \B^1_{\textup{even}}+\dim \ker \B^2_{\textup{even}}-\dim \ker \B_{\textup{even}}\right),
\end{align*}
where $\B$ and $\B^j, j=1,2$ are the odd-signature operators associated to the Fredholm complexes $(\domm, \DD)$ and $(\domm_j,\DD_j),j=1,2$ respectively. Further $P,P_1$ denote the boundary conditions and the Calderon projector associated to $\B^1_{\textup{even}}$, respectively. $\tau_{\mu}$ is the Maslov triple index.
\\[3mm] The $\textup{sign}\in \{\pm 1\}$ is a combinatorial sign, explicitly defined in \eqref{sign}. 
\end{thm} \ 
\\ \
\\[-12mm]
\begin{cor}\textup{[Gluing formula for scalar analytic torsion]}
$$
\frac{T^{RS}(M,E)}{T^{RS}_{\textup{rel}}(M_1,E)\cdot T^{RS}_{\textup{abs}}(M_2,E)}=\tau (\mathcal{H})^{-1}\cdot 2^{\chi(N)/2}.$$
\end{cor}
\begin{proof}
Comparison of the statement of Corollary \ref{phase} with the relation \eqref{relation} we obtain along the result of Theorem \ref{gluing-final} the following formula as a byproduct:
$$\frac{T^{RS}(\domm,\DD)}{T^{RS}(\domm_1,\DD_1)T^{RS}(\domm_2,\DD_2)}=\tau (\mathcal{H})^{-2}\cdot 2^{\chi(N)}.$$
By construction and the Poincare duality on odd-dimensional manifolds with or without boundary \eqref{PD-laplace} we know
\begin{align*}
T^{RS}(\domm_1,\DD_1)=T^{RS}_{\textup{rel}}(M_1,E)\cdot T^{RS}_{\textup{abs}}(M_1,E)= T^{RS}_{\textup{rel}}(M_1,E)^2, \\
T^{RS}(\domm_2,\DD_2)=T^{RS}_{\textup{abs}}(M_2,E)\cdot T^{RS}_{\textup{rel}}(M_2,E)= T^{RS}_{\textup{abs}}(M_2,E)^2, \\
T^{RS}(\domm,\DD)=T^{RS}(M,E)^2.
\end{align*}
Taking now square-roots gives the result.
\end{proof}\ \\
\\[-7mm] Note that this result refines the result of [Lee, Theorem 1.7 (2)] on the adiabatic decomposition of the scalar analytic torsion.
\\[3mm] Note finally that in view of [BV3, Theorem 5.3] the gluing formula in Theorem \ref{gluing-final} can be viewed as a gluing formula for refined analytic torsion in the version of Braverman-Kappeler.

\section{References}\
\\[-1mm] [Ag] S. Agmon \emph{"On the eigenfunctions and on the eigenvalues of general elliptic boundary value problems"} Comm. Pure Appl. Math., vol. 15, 119-147 (1962)
\\[3mm] [APS] M. F. Atiyah, V.K. Patodi, I.M. Singer \emph{"Spectral asymmetry and Riemannian geometry I"}, Math. Proc. Camb. Phil. Soc. 77, 43-69 (1975)
\\[3mm] [BFK] D. Burghelea, L. Friedlander, T. Kappeler \emph{"Mayer-Vietoris type formula for determinants of elliptic differential operators"}, Journal of Funct. Anal. 107, 34-65 (1992)
\\[3mm] [BGV] N. Berline, E. Getzler, M. Vergne \emph{"Heat kernels and Dirac operators"}, Springer-Verlag, New Jork (1992)
\\[3mm] [BK1] M. Braverman and T. Kappeler \emph{"Refined analytic torsion"}, arXiv:math.DG/ 0505537v2, to appear in J. of. Diff. Geom.
\\[3mm] [BK2] M. Braverman and T. Kappeler \emph{"Refined Analytic Torsion as an Element of the Determinant Line"} , arXiv:math.GT/0510532v4, To appear in Geometry \& Topology 
\\[3mm] [BL1] J. Br\"{u}ning, M. Lesch \emph{"Hilbert complexes"}, J. Funct. Anal. 108, 88-132 (1992)
\\[3mm] [BL3] J. Br\"{u}ning, M. Lesch \emph{"On boundary value problems for Dirac type operators. I. Regularity and self-adjointness"}, arXiv:math/9905181v2 [math.FA] (1999)
\\[3mm] [BLZ] B. Booss, M. Lesch, C. Zhu \emph{"The Calderon Projection: New Definition and Applications"}, arXiv:math.DG/0803.4160v1 (2008)
\\[3mm] [BM] J. Br\"{u}ning, X. Ma \emph{"An anomaly-formula for Ray-Singer metrics on manifolds with boundary"}, Geom. Funct. An. 16, No. 4, 767-837, (2006)
\\[3mm] [BS] J. Br\"{u}ning, R. Seeley \emph{"An index theorem for first order regular singular operators"}, Amer. J. Math 110, 659-714, (1988)
\\[3mm] [BW] B. Booss, K. Wojchiechovski \emph{"Elliptic boundary problemsfor Dirac Operators"}, Birkh\"{a}user, Basel (1993)
\\[3mm] [BV1] B. Vertman \emph{"Functional Determinants for Regular-Singular Laplace-type Operators"}, preprint, arXiv:0808.0443 (2008) 
\\[3mm] [BV3] B. Vertman \emph{"Refined Analytic Torsion on Manifolds with Boundary"}, preprint, arXiv:0808.0416 (2008) 
\\[3mm] [BZ] J. -M. Bismut and W. Zhang \emph{"Milnor and Ray-Singer metrics on the equivariant determinant of a flat vector bundle"}, Geom. and Funct. Analysis 4, No.2, 136-212 (1994)
\\[3mm] [BZ1] J. -M. Bismut and W. Zhang \emph{"An extension of a Theorem by Cheeger and M\"{u}ller"}, Asterisque, 205, SMF, Paris (1992)
\\[3mm] [Ch] J. Cheeger \emph{"Analytic Torsion and Reidemeister Torsion"}, Proc. Nat. Acad. Sci. USA 74 (1977), 2651-2654
\\[3mm] [Fa] M. Farber \emph{Combinatorial invariants computing the Ray-Singer analytic torsion}, arXiv:dg-ga/ 9606014v1 (1996)
\\[3mm] [Gi] P.B. Gilkey \emph{"Invariance Theory, the Heat-equation and the Atiyah-Singer Index Theorem"}, Second Edition, CRC Press (1995)
\\[3mm] [Gi2] P.B. Gilkey \emph{"The eta-invariant and secondary characteristic classes of locally flat bundles"}, Algebraic and Differential topology $-$ global differential geometry, Teubner-Texte zur Math., vol. 70, Teubner, Leipzig, 49-87 (1984)
\\[3mm] [GS1] P. Gilkey, L. Smith \emph{"The eta-invariant for a class of elliptic boundary value problems"}, Comm. Pure Appl. Math. Vol.36, 85-131 (1983)
\\[3mm] [GS2] P. Gilkey, L. Smith \emph{"The twisted index problem for manifolds with boundary"}, J. Diff. Geom. 18, 393-444 (1983)
\\[3mm] [K] T. Kato \emph{"Perturbation Theory for Linear Operators"}, Die Grundlehren der math. Wiss. Volume 132, Springer (1966)
\\[3mm] [KL] P. Kirk and M. Lesch \emph{"The $\eta$-invariant, Maslov index and spectral flow for Dirac type operators on manifolds with boundary"}, Forum Math. 16, 553-629 (2004)
\\[3mm] [KM] F.F: Knudsen, D. Mumford \emph{"The projectivity of the moduli space of stable curves. I. Preliminaries on 'det' and 'Div'"}, Math. Scand. 39, no1, 19-55 (1976)
\\[3mm] [KN] S. Kobayashi, K. Nomizu \emph{"Foundations of differential geometry"}, Volume I, Interscience Publishers (1963)
\\[3mm] [L2] M. Lesch \emph{"Gluing formula in cohomological algebra"}, unpublished notes.
\\[3mm] [Lee] Y. Lee \emph{"Burghelea-Friedlander-Kappeler's gluing formula for the zeta-determinant and its application to the adiabatic decomposition of the zeta-determinant and the analytic torsion"}, Trans. Amer. Math. Soc., Vol. 355, 10, 4093-4110 (2003)
\\[3mm] [LR] J. Lott, M. Rothenberg \emph{"Analytic torsion for group actions"} J. Diff. Geom. 34, 431-481 (1991)
\\[3mm] [L\"{u}] W. L\"{u}ck \emph{"Analytic and topological torsion for manifolds with boundary and symmetry"}, J. Diff. Geom. 37, 263-322, (1993) 
\\[3mm] [Mi] J. Milnor \emph{"Whitehead torsion"}, Bull. Ams. 72, 358-426 (1966)
\\[3mm] [Mu] W. M\"{u}ller \emph{"Analytic torsion and R-torsion for unimodular representations"} J. Amer. Math. Soc., Volume 6, Number 3, 721-753 (1993) 
\\[3mm] [Mu1] W. M\"{u}ller \emph{"Analytic Torsion and R-Torsion of Riemannian manifolds"} Adv. Math. 28, 233-305 (1978)
\\[3mm] [Mun] J. Mukres \emph{"Elementary differential topology"} Ann. of Math. Stud. vol. 54, Princeton Univ. Press, Princeton, NJ (1961)
\\[3mm] [MZ1] X. Ma, W. Zhang \emph{"$\eta -$invariant and flat vector bundles I"}, Chinese Ann. Math. 27B, 67-72 (2006)
\\[3mm] [MZ2] X. Ma, W. Zhang \emph{"$\eta -$invariant and flat vector bundles II"}, Nankai Tracts in Mathematics. Vol. 11. World Scientific, 335-350, (2006)
\\[3mm] [Nic] L.I. Nicolaescu \emph{"The Reidemeister torsion of 3-manifolds"}, de Gruyter Studies in Mathematics, vol. 30, Berlin (2003)
\\[3mm] [Re1] K. Reidemeister \emph{"Die Klassifikation der Linsenr\"{a}ume"}, Abhandl. Math. Sem. Hamburg 11, 102-109 (1935)
\\[3mm] [Re2] K. Reidemeister \emph{"\"{U}berdeckungen von Komplexen"}, J. reine angew. Math. 173, 164-173 (1935)
\\[3mm] [ReS] M. Reed, B. Simon \emph{"Methods of Mathematical Physics"}, Vol. II, Acad.N.J. (1979)
\\[3mm] [Rh] G. de Rham \emph{"Complexes a automorphismes et homeomorphie differentiable"}, Ann. Inst. Fourier 2, 51-67 (1950)
\\[3mm] [RS] D.B. Ray and I.M. Singer \emph{"R-Torsion and the Laplacian on Riemannian manifolds"}, Adv. Math. 7, 145-210 (1971)
\\[3mm] [Ru] W. Rudin \emph{"Functional Analysis"}, Second Edition, Mc. Graw-Hill, Inc. Intern. Series in pure and appl. math. (1991)
\\[3mm] [RH] Rung-Tzung Huang \emph{"Refined Analytic Torsion: Comparison theorems and examples"}, math.DG/0602231v2 
\\[3mm] [Se1] R. Seeley \emph{"The resolvent of an elliptic boundary problem"} Amer. J. Math. 91 889-920 (1969)
\\[3mm] [Se2] R. Seeley \emph{"An extension of the trace associated with elliptic boundary problem"} Amer. J. Math. 91 963-983 (1969)
\\[3mm] [Sh] M.A. Shubin \emph{"Pseudodifferential operators and Spectral Theory"}, English translation: Springer, Berlin (1086) 
\\[3mm] [Tu1] V. G. Turaev \emph{"Euler structures, non-singular vector-fields and torsion of Reidemeister type"}, English Translation: Math. USSR Izvestia 34:3 627-662 (1990)
\\[3mm] [Tu2] V. G. Turaev \emph{"Torsion invariants of Spin$^c$ structures on three-manifolds"}, Math. Research Letters 4:5 679-695 (1997)
\\[3mm] [V] S. Vishik \emph{"Generalized Ray-Singer Conjecture I. A manifold with smooth boundary"}, Comm. Math. Phys. 167, 1-102 (1995)
\\[3mm] [Wh] J. H. Whitehead \emph{"Simple homotopy types"}, Amer. J. Math. 72, 1-57 (1950)

\end{document}